\documentclass[12pt]{amsart}
\usepackage{amssymb}
\usepackage[all]{xypic}

\renewcommand{\labelenumi}{(\alph{enumi})}

\theoremstyle{plain}
\newtheorem{theorem}{Theorem}[section]
\newtheorem{lemma}[theorem]{Lemma}
\newtheorem{corollary}[theorem]{Corollary}
\newtheorem{prop}[theorem]{Proposition}
\newtheorem{conj}[theorem]{Conjecture}

\theoremstyle{remark}

\newtheorem{remark}[theorem]{Remark}

\newtheorem*{note*}{Note}
\newtheorem*{remark*}{Remark}
\newtheorem*{example*}{Example}

\theoremstyle{definition}
\newtheorem*{definition*}{Definition}
\newtheorem{definition}[theorem]{Definition}

\newcommand{\Z}{\mathbb{Z}}
\newcommand{\R}{\mathbb{R}}
\newcommand{\Q}{\mathbb{Q}}
\newcommand{\C}{\mathbb{C}}

\newcommand{\Gal}{\mathrm{Gal}}

\newcommand{\Norm}{\mathrm{Norm}}
\newcommand{\cl}{\mathrm{cl}}

\newcommand{\tr}{\mathrm{Tr}}

\newcommand{\Frakp}{\mathfrak{P}}

\newcommand{\frakp}{\mathfrak{p}}

\newcommand{\Ann}{\mathrm{Ann}}

\newcommand{\euler}{\chi_{\mathrm{ref}}}
\newcommand{\nr}{\mathrm{nr}}
\newcommand{\onr}{\overline{\mathrm{nr}}}
\newcommand{\Mat}{\mathrm{Mat}}
\newcommand{\fit}{\mathrm{Fit}_{\mathcal{O}}}
\newcommand{\ind}{\mathrm{ind}_{G_{\mathfrak{p}}}^{G}}
\newcommand{\hg}{h_{\mathrm{glob}}}
\newcommand{\hgi}{h^{-1}_{\mathrm{glob}}}
\newcommand{\pr}{\mathrm{pr}}
\newcommand{\im}{\mathrm{im}}

\renewcommand{\det}{\mathrm{det}}

\title[A non-abelian Stickelberger theorem]{A non-abelian Stickelberger theorem}

\author{David Burns}
\address{David Burns\\
King's College London\\
Dept.\ of Mathematics\\
London WC2R 2LS\\
U.K.
}
\email{david.burns@kcl.ac.uk}
\urladdr{http://www.mth.kcl.ac.uk/staff/d\_burns.html}

\author{Henri Johnston}
\address{Henri Johnston\\
St.\ John's College\\
St.\ John's Street\\
Cambridge CB2 1TP\\
U.K.
}
\email{H.Johnston@dpmms.cam.ac.uk}
\urladdr{http://www.dpmms.cam.ac.uk/$\sim$hlj31/}

\date{Version of 10th March 2010}

\begin{document}

\begin{abstract}
Let $L/k$ be a finite Galois extension of number fields with Galois group $G$.
For every odd prime $p$ satisfying certain mild technical hypotheses, we
use values of Artin $L$-functions to construct an element in the centre of the
group ring $\Z_{(p)}[G]$ that annihilates the $p$-part of the class group of $L$.
\end{abstract}

\maketitle

\section{Introduction and statement of the main results}\label{intro}

Let $K/k$ be a finite Galois extension of number fields with Galois group $G$.
Let $\mathcal{S}$ be a finite set of places of $k$ containing the infinite places $\mathcal{S}_{\infty}$.
For any (complex) character $\chi$ of $G$, we let
$e_{\chi} = \frac{\chi(1)}{|G|}\sum_{g \in G} \chi(g^{-1})g$ denote the corresponding central idempotent
of the group algebra $\C[G]$ and let $L_{\mathcal{S}}(s, \chi)$ denote the truncated
Artin $L$-function attached to $\chi$ and $\mathcal{S}$. Summing over all irreducible characters of $G$ gives a so-called `Stickelberger element'
\[
\Theta(K/k, \mathcal{S}) :=
\sum_{\chi \in \mathrm{Irr}(G)} L_{\mathcal{S}}(0, \bar{\chi}) \cdot e_{\chi}.
\]

Now suppose that $k$ is totally real, $K$ is a CM field, $G$ is abelian and
$\mathcal{S}$ contains the ramified places
$\mathcal{S}_{\mathrm{ram}}(K/k)$.
Let $\mu_{K}$ denote the roots of unity in $K$ and let $\cl_{K}$ denote the
class group of $K$.
In \cite{MR524276} and \cite{MR579702}
independently it was shown that
\[
\Ann_{\Z[G]}(\mu_{K}) \Theta(K/k, \mathcal{S}) \subseteq \Z[G].
\]
It is now easy to state Brumer's conjecture, which can be seen as a generalisation of Stickelberger's theorem.

\begin{conj}\label{conj:brumer}
In the above situation, $\Ann_{\Z[G]}(\mu_{K}) \Theta(K/k, \mathcal{S})$ annihilates $\cl_{K}$.
\end{conj}

There is a large body of evidence in support of Brumer's conjecture;
see the expository article \cite{MR2088712}, for example.
Furthermore, under the assumptions that the appropriate special case of the equivariant Tamagawa number conjecture (ETNC) holds (see \S \ref{sec:tate-ETNC}) and the non-$2$-part of $\mu_{K}$ is a cohomologically trivial $G$-module, Greither has shown that Brumer's conjecture holds outside the $2$-part (see \cite{MR2371374}).

By contrast, as far as we are aware, there is still no Brumer-type annihilation result proved for any non-abelian extension. In the present article, we address this situation by proving an unconditional annihilation result for arbitrary (not necessarily abelian) extensions, from which a weak form of Brumer's conjecture can also be deduced.

Before stating the main result, we introduce some additional notation.
For any natural number $n$, we let $\zeta_{n}$ denote a primitive $n$th root of unity.
For any number field $F$, we write $F^{\mathrm{cl}}$
for the normal closure and $F^{+}$ for the maximal totally real subfield of $F$.
For a complex character $\chi$ of a finite group $G$, we let
$E=E_{\chi}$ denote a subfield of $\C$ over which $\chi$ can be realised
that is both Galois and of finite degree over $\Q$, and let $\mathcal{O}=\mathcal{O}_{E}$ denote the
ring of algebraic integers of $E$. Furthermore, we write $\pr_{\chi}$ for the associated `projector'
$\sum_{g \in G} \chi(g^{-1}) g$ in the group algebra $E[G]$ and
$\mathcal{D}_{E/\Q}$ for the different of the extension $E/\Q$.

We shall postpone two important definitions until \S \ref{precise}; for the moment we shall
only give brief descriptions. In the theorem below, $U_{\chi}$ is an explicit fractional ideal of
$\mathcal{O}$ depending on $\mathcal{S}_{\mathrm{ram}}(K/k)$ (it is often the case that
$U_{\chi}$ is trivial; see Remark \ref{rmk:simplifying-thm}(ii) and \S \ref{sec:U-chi}).
The fractional $\mathcal{O}$-ideal $h(\mu_{K},\chi)$ is a natural truncated Euler characteristic 
of the $\chi$-twist of $\mu_{K}$. We abbreviate $L_{\mathcal{S}_{\infty}}(s, \chi)$ to $L(s, \chi)$.

\begin{theorem}\label{main-theorem}
Let $L/k$ be a finite Galois extension of number fields with Galois group $G$.
Fix a non-trivial irreducible character $\chi$ of $G$.
Let $K:=L^{\ker(\chi)}$ be the subfield of $L$ cut out by $\chi$.
Let $p$ be any odd prime satisfying the following condition:
\begin{itemize}
\item[(\textasteriskcentered)] If
\emph{(a)} $k$ is totally real,
\emph{(b)} $K$ is a CM field, and
\emph{(c)} $K^{\mathrm{cl}} \subset (K^{\mathrm{cl}})^{+}(\zeta_{p})$, \\
then no prime of $K^{+}$ above $p$ is split in $K/K^{+}$.
\end{itemize}
Then for any element $x$ of $\mathcal{D}_{E/\Q}^{-1} \cdot h(\mu_{K},\chi) \cdot U_{\chi}$,
the sum
\[
\sum_{\omega \in \Gal(E_{\chi}/\Q)} \!\!\! x^{\omega} L(0, \bar{\chi}^{\omega})
\cdot \pr_{\chi^{\omega} }
\]
belongs to the centre of $\Z_{(p)}[G]$ and annihilates $\Z_{(p)}  \otimes_{\Z} \cl_L$.
\end{theorem}

\begin{remark}\label{rmk:simplifying-thm}
The statement of Theorem \ref{main-theorem} can be simplified in several cases:
\renewcommand{\labelenumi}{(\roman{enumi})}
\begin{enumerate}
\item If an odd prime $p$ is unramified in $K/\Q$ then (b) forces
$K^{\mathrm{cl}} \not \subset (K^{\mathrm{cl}})^{+}(\zeta_{p})$, so condition
(\textasteriskcentered) holds trivially. Furthermore, if $k$ is normal over
$\Q$ and $[(K^{\mathrm{cl}})^{+}:\Q]$ is odd then condition (\textasteriskcentered)
holds for all odd primes $p$ (these hypotheses together with (a), (b) and (c)
imply that all the primes of $K^{+}$ above $p$ are in fact ramified in $K/K^{+}$).

\item If every inertia subgroup of $\Gal(K/k)$ is normal
(for example, $\chi$ is linear or every
$\frakp \in \mathcal{S}_{\mathrm{ram}}(K/k)$ is non-split in $K/k$),
then under the assumption that $\chi$ is non-trivial and irreducible
it is straightforward to show that $U_{\chi}$ is trivial (see \S \ref{sec:U-chi}).

\item If an odd prime $p$ does not divide $|\mu_{K}|$ then $h(\mu_{K},\chi)$ is
relatively prime to $p$, and so this term can be ignored. In particular, this is the case if
$p$ is unramified in $K/\Q$.
\end{enumerate}
\end{remark}

\begin{remark}
The purpose of condition (\textasteriskcentered) is to ensure that when (a) and (b) hold, the Strong Stark Conjecture at $p$ as formulated by Chinburg in \cite[Conjecture 2.2]{MR724009} holds for the (odd) character $\chi$.
Hence condition (\textasteriskcentered) can be ignored completely in each of the following
cases in which the Strong Stark Conjecture is already known to be valid:
\renewcommand{\labelenumi}{(\roman{enumi})}
\begin{enumerate}
\item $\chi$ is rational valued: this was proved by Tate in \cite[Chapter II, Theorem 8.6]{MR782485};
\item $k=\Q$ and $\chi$ is linear: this was proved by Ritter and Weiss in \cite{MR1423032}
(in fact they show that the conjecture holds if $2$ is unramified in $K/\Q$ or holds outside the $2$-part otherwise, but this is all we need as $p$ is odd);
\item $k$ is an imaginary quadratic field of class number one and $\chi$ is a linear
character whose order is divisible only by primes which split completely in $k / \Q$: this
follows from (\cite[\S 3]{MR1981031} and) the result of Bley in \cite[Theorem 4.2]{MR2226270}.
\end{enumerate}
Note that in particular, we are in case (i) if $G$ is isomorphic to
the symmetric group on any number of elements,
the quaternion group of order $8$,
or any direct product of such groups.
\end{remark}

We give the proof of the following corollary in \S \ref{sec:main-results} after the proof of Theorem \ref{main-theorem}.

\begin{corollary}\label{cor:weak-brumer}
Let $L/k$ be a finite Galois extension of number fields with Galois group $G$.
Suppose that every inertia subgroup is normal in $G$ (for example,
every $\frakp \in \mathcal{S}_{\mathrm{ram}}(L/k)$ is non-split in $L/k$.)
Let $\mathcal{S}$ be any finite set of places of $k$ containing the infinite places 
$\mathcal{S}_{\infty}$. 
For any irreducible character $\chi$ of $G$, let $\Q(\chi)$ denote the character field of $\chi$
and let $d_{\chi}$ be the minimum of $[E_{\chi}:\Q(\chi)]$ over all possible choices of $E_{\chi}$. Let p be any odd prime that is unramified in $L/\Q$.
Then the element
\[
\sum_{\chi \in \mathrm{Irr}(G), \chi \neq 1} \!\!\!
L_{\mathcal{S}}(0, \bar{\chi}) \cdot d_{\chi} \pr_{\chi}
\]
belongs to the centre of $\Z_{(p)}[G]$ and annihilates $\Z_{(p)} \otimes_{\Z} \cl_L$.
\end{corollary}

\begin{remark}
Note that $E_{\chi}$ can always be taken to be $\Q(\zeta_{n})$ where $n$ is the exponent of $G$
(see \cite[(15.18)]{MR632548}),
and so $[\Q(\zeta_{n}):\Q(\chi)]$ is an upper bound for $d_{\chi}$. In fact, $d_{\chi}=1$ whenever $G$ is abelian or of odd prime power order, or is isomorphic to the symmetric group on any number of letters, the dihedral group of any order, or any direct product of such groups.
\end{remark}

\begin{remark}
Suppose that $k$ is totally real, $L$ is a CM field, $G$ is abelian and
$\mathcal{S}$ contains the ramified primes $\mathcal{S}_{\mathrm{ram}}(L/k)$.
Then Corollary \ref{cor:weak-brumer} says that for $p$ odd and unramified in $L/\Q$,
\[
\sum_{\chi \in \mathrm{Irr}(G), \chi \neq 1} \!\!\!
L_{\mathcal{S}}(0, \bar{\chi}) \cdot \pr_{\chi}
= |G| \cdot \Theta(L/k, \mathcal{S})
\quad \textrm{ annihilates } \quad
\Z_{(p)} \otimes_{\Z} \cl_{K}.
\]
(Note that there is a slight adjustment to be made in the case $k=\Q$.)
Under the hypotheses on $p$ we have $\Ann_{\Z[G]}(\mu_{L}) \otimes_{\Z} \Z_{(p)} = \Z_{(p)}[G]$
and so the above is the same statement as the `$p$-part' of Brumer's conjecture
(Conjecture \ref{conj:brumer}) but with an extra factor of $|G|$ in the annihilator
(of course, this makes no difference if $p$ does not divide $|G|$).
\end{remark}

\begin{remark}
In \cite{DerLfncs} a conjecture is given which, in the setting of present article, 
uses values of derivatives of Artin $L$-functions to construct explicit annihilators of ideal class groups.
Upon restriction to the abelian case and to consideration of values of Artin $L$-functions, the central conjecture of \cite{DerLfncs} precisely recovers Brumer's conjecture. \cite{DerLfncs} also studies explicit examples which show that, in some cases, Theorem \ref{main-theorem} is essentially the strongest possible annihilation result.
\end{remark}

\section{Definition of $U_{\chi}$ and $h(\mu_{K},\chi)$}\label{precise}

In this section we give the necessary background material to make precise the definitions of
$U_{\chi}$ and $h(\mu_{K},\chi)$ of Theorem \ref{main-theorem}.

\subsection{$\chi$-twists}

We largely follow the exposition of \cite[\S 1]{MR2443986}.
Fix a finite group $G$ and an irreducible (complex) character $\chi$ of $G$.
Let $E=E_{\chi}$ be a subfield of $\C$ over which $\chi$ can be realised
that is both Galois and of finite degree over $\Q$.
We write $\mathcal{O}$ for the ring of algebraic integers in $E$ and set
\[
e_{\chi} := \frac{\chi(1)}{|G|}\sum_{g \in G}\chi(g^{-1})g = 
\frac{\chi(1)}{|G|}\sum_{g \in G}\overline{\chi}(g)g, \qquad
\pr_{\chi} :=  \frac{|G|}{\chi(1)} e_{\chi} = \sum_{g \in G}\chi(g^{-1})g
= \sum_{g \in G}\overline{\chi}(g)g.
\]
Here $e_{\chi}$ is a primitive central idempotent of $E[G]$ and $\pr_{\chi}$
is the associated `projector'. 

We choose a maximal $\mathcal{O}$-order $\mathfrak{M}$ in $E[G]$ containing $\mathcal{O}[G]$ and
fix an indecomposable idempotent $f_{\chi}$ of $e_{\chi}\mathfrak{M}$.
We define an $\mathcal{O}$-torsion-free right $\mathcal{O}[G]$-module
by setting $T_{\chi} := f_{\chi}\mathfrak{M}$.
(Note that this is slightly different from the definition given in \cite[\S 1]{MR2443986}.)
The associated right $E[G]$-module $E \otimes_{\mathcal{O}} T_{\chi}$ has character 
$\chi$ and $T_{\chi}$ is locally free of rank $\chi(1)$ over $\mathcal{O}$.

For any (left) $G$-module $M$ we set $M[\chi]:=T_{\chi} \otimes_{\Z} M$, upon which $G$
acts on the left by $t \otimes_{\Z} m \mapsto tg^{-1} \otimes_{\Z} g(m)$ for each
$t \in T_{\chi}, m \in M$ and $g \in G$. For any $G$-module $M$
and integer $i$ we write $\widehat{H}^{i}(G,M)$ for the Tate cohomology in degree $i$ of $M$ with
respect to $G$. We also write $M^{G}$ for the maximal submodule,
respectively $M_{G}$ for maximal quotient module, of $M$ upon which $G$ acts trivially. Then we obtain a left exact functor $M \mapsto M^{\chi}$, respectively right exact functor $M \mapsto M_{\chi}$, from the category of left $G$-modules to the category of $\mathcal{O}$-modules by setting $M^{\chi} := M[\chi]^{G}$ and $M_{\chi} := M[\chi]_{G} = T_{\chi} \otimes_{\Z[G]} M$.
The action of
$\Norm_{G} := \sum_{g \in G} g$ on $M[\chi]$ induces a homomorphism of $\mathcal{O}$-modules
$t(M, \chi): M_{\chi} \rightarrow M^{\chi}$ with kernel $\widehat{H}^{-1}(G, M[\chi])$ and cokernel
$\widehat{H}^{0}(G, M[\chi])$. Thus $t(M, \chi)$ is bijective whenever $M$, and hence also $M[\chi]$, is a cohomologically trivial $G$-module.

We shall henceforth take `module' to mean `left module' unless explicitly stated otherwise.

\subsection{Reducing to the case $L=K$}\label{reduce-to-L=K}

Assume the setting and notation of Theorem \ref{main-theorem} for the rest of this section. In the definitions of $U_{\chi}$ and $h(\mu_{K},\chi)$ below, we shall assume that $L=K$.
Hence $\chi$ is a non-trivial irreducible faithful character of $G=\Gal(K/k)$.

For the general case $L \neq K$, let $\phi$ be the character of $\Gal(K/k)$ that inflates
to $\chi$. Then $\phi$ is irreducible and faithful, and we have $E_{\chi}=E_{\phi}$.
We define $U_{\chi} := U_{\phi}$ and $h(\mu_{K},\chi) := h(\mu_{K},\phi)$.

\subsection{Definition of $U_{\chi}$}\label{sec:U-chi}

We first recall the following construction from \cite[\S 2]{MR2371374}.
Let $\frakp$ be a finite prime of $k$ and fix a prime $\Frakp$ of $K$ above $\frakp$.
We use the standard notation $G_{\frakp}, G_{0, \frakp}$ and
$\overline{G}_{\frakp} = G_{\frakp} / G_{0, \frakp}$ for, respectively,
the decomposition group, inertia group and the residual group of
$K/k$ at $\Frakp$. Choose a lift $F_{\frakp}$ (fixed for the rest of the paper)
of the Frobenius element ${\rm Fr_{\frakp}} \in \overline{G}_{\frakp}$ to $G_{\frakp} \subset G$.
For any subgroup $H$ of $G$, let $\Norm_{H} := \sum_{h \in H} h$.
We define central idempotents of $\Q[G_{\frakp}]$ as follows:
\[
\begin{array}{ll}
e_{\frakp}' := | G_{0, \frakp}|^{-1} \Norm_{G_{0, \frakp}}, & e_{\frakp}'' := 1 - e_{\frakp}'; \\
\bar{e}_{\frakp} :=  | G_{\frakp}|^{-1} \Norm_{G_{\frakp}}, & \bar{\bar{e}}_{\frakp} := 1 - \bar{e}_{\frakp}.
\end{array}
\]
We define the $\Z[G_{\frakp}]$-modules $U_{\frakp}$ by
\[
U_{\frakp} := \langle \Norm_{G_{0,\frakp}}, 1 - e_{\frakp}'F_{\frakp}^{-1} \rangle_{\Z[G_{\frakp}]} \subset
\Q[G_{\frakp}],
\]
and note that $U_{\frakp} = \Z[G_{\frakp}]$ if $\frakp$ is unramified in $K/k$.

Let $\nr_{e_{\chi}E[G]}:e_{\chi}E[G] \rightarrow E$ be the reduced
norm map (see \cite[\S 7D]{MR632548}). More explicitly, this is
the determinant map $e_{\chi} E[G] \cong \Mat_{\chi(1)}(E)
\rightarrow E$. 
We define a fractional ideal of $\mathcal{O}$ by setting
\[
U_{\chi} :=  \prod_{\frakp \in \mathcal{S}_{\mathrm{ram}}(K/k)}
\nr_{e_{\chi}E[G]}(e_{\chi}\mathfrak{M}U_{\frakp})\mathcal{O}.
\]
Here we use the following notation: for any finitely generated
$\Z[G_{\frakp}]$-submodule $V_{\frakp}$ of $\Q[G_\frakp]$ we write
$\nr_{e_{\chi}E[G]}(e_{\chi}\mathfrak{M}V_{\frakp})\mathcal{O}$
for the $\mathcal{O}$-submodule of $E$ that is generated
by the elements $\nr_{e_{\chi}E[G]}(x)$ as $x$ runs over 
$e_{\chi}\mathfrak{M}V_{\frakp}$
and note that this is indeed a fractional ideal of $\mathcal{O}$
since $\nr_{e_{\chi}E[G]}(e_{\chi}\mathfrak{M}) = \mathcal{O}$.

Recalling the hypothesis that $\chi$ is faithful and non-trivial,
it is a straightforward exercise to show that $U_{\chi}$ is the
trivial ideal if $G_{0,\frakp}$ is normal in $G$ for every $\frakp
\in \mathcal{S}_{\mathrm{ram}}(K/k)$ (the point is that $\chi$ must
be non-trivial on $G_{0,\frakp}$ and so $e_{\chi}$ annihilates both
$\Norm_{G_{0,\frakp}}$ and $e_{\frakp}'$). In particular, this is
the case if $G$ is abelian or Hamiltonian (i.e. every subgroup of $G$ is
normal), or every $\frakp \in \mathcal{S}_{\mathrm{ram}}(K/k)$ is
non-split in $K/k$ (i.e. $G=G_{\frakp}$).

\subsection{Definition of $h(\mu_{K},\chi)$}
For any finitely generated $\mathcal{O}$-module $M$, we let $\fit(M)$ denote
the Fitting ideal of $M$. We define $h(\mu_{K},\chi)$ to be
the natural truncated Euler characteristic
\[
h(\mu_{K},\chi) := \prod_{i=0}^{i=2} \fit(H^{i}(G, \mu_{K}[\chi]))^{(-1)^{i}}.
\]
Note that if $\mu_{K}$ is cohomologically trivial as a $G$-module
then $h(\mu_{K},\chi)=\fit(\mu_{K}[\chi]^{G})$.

\section{Algebraic $K$-theory}\label{nrbmkt}

In this section we summarise some of the necessary background material from algebraic $K$-theory.
Further details can be found in \cite{MR632548}, \cite{MR892316}, \cite[\S 2]{MR2371375} and
\cite[\S 2]{breuning-thesis}.

\subsection{Relative $K$-theory}
For any integral domain $R$ of characteristic $0$, any extension field $F$ of the field of
fractions of $R$ and any finite group $G$, let $K_{0}(R[G], F[G])$ denote the relative
algebraic $K$-group associated to the ring homomorphism $R[G] \hookrightarrow F[G]$.
We write $K_{0}(R[G])$ for the Grothendieck group of the category of finitely generated
projective $R[G]$-modules and $K_{1}(R[G])$ for the Whitehead group. There is a long
exact sequence of relative $K$-theory
\begin{equation}\label{long-exact-seq}
K_{1}(R[G]) \longrightarrow K_{1}(F[G]) \longrightarrow K_{0}(R[G],F[G])
\longrightarrow K_{0}(R[G]) \longrightarrow K_{0}(F[G]).
\end{equation}

\subsection{Reduced norms}
Let $\zeta(F[G])^{\times}$ denote the multiplicative group of the centre of $F[G]$.
There exists a reduced norm map $\nr_{F[G]} : (F[G])^{\times} \rightarrow \zeta(F[G])^{\times}$ whose
image is denoted by $\zeta(F[G])^{\times +}$ and there is a natural surjective map
$(F[G])^{\times} \rightarrow K_{1}(F[G])$, $x \mapsto (F[G], x_{r})$, where $x_{r}$ denotes
right multiplication by $x$. However, these maps have the same kernel, namely
the commutator subgroup $[(F[G])^{\times},(F[G])^{\times}]$, and so we have the following
commutative diagram
\[
\xymatrix@1@!0@=36pt {
(F[G])^{\times}  \ar@{->}[d]_{\nr_{F[G]}} \ar@{->}[rr] & & K_{1}(F[G])
\ar@{-->}[dll]^{\onr_{F[G]}}_{\simeq} \\
\zeta(F[G])^{\times +} }
\]
where $\onr_{F[G]}$ is the induced isomorphism. Note that the inverse map
$\onr_{F[G]}^{-1}:\zeta(F[G])^{\times +} \rightarrow K_{1}(F[G])$ can be described
explicitly by $\nr_{F[G]}(x) \mapsto (F[G], x_{r})$.
By composing the map $\onr_{F[G]}^{-1}: \zeta(F[G])^{\times +} \rightarrow K_{1}(F[G])$
with the boundary map $K_{1}(F[G]) \rightarrow K_{0}(R[G], F[G])$,
we therefore obtain a homomorphism
\begin{equation}\label{eqn:explicit-partial-hom}
\partial_{R[G], F[G]}: \zeta(F[G])^{\times +} \longrightarrow  K_{0}(R[G], F[G]),
\quad \nr_{F[G]}(x) \mapsto (R[G], x_{r}, R[G]).
\end{equation}
We note that if $F$ is algebraically closed then $\nr_{F[G]}$ is surjective, i.e.,
$\zeta (F[G])^{\times +} = \zeta (F[G])^{\times}$.
In any case, we always have $(\zeta(F[G])^{\times})^{2} \subseteq \zeta(F[G])^{\times +}$.

\subsection{Induction}\label{subsec:induction}

Let $H$ be a subgroup of $G$.
The functor $M \mapsto R[G] \otimes_{R[H]} M$ from projective $R[H]$-modules
to projective $R[G]$-modules and the corresponding functor from $F[H]$-modules to $F[G]$-modules induce induction maps $\mathrm{ind}^{G}_{H}$ for all $K$-groups in the exact sequence
\eqref{long-exact-seq}. We also obtain an induction map
$i_{H}^{G}:= \onr_{F[G]} \circ \mathrm{ind}^{G}_{H}\circ \onr_{F[H]}^{-1}: \zeta(F[H])^{\times +} \rightarrow \zeta(F[G])^{\times +}$.

Specialising to the case $R=\Z$ and $F=\R$, we
have the following commutative diagram
\begin{equation}\label{eqn:K-diagram}
\xymatrix@1@!0@=36pt {
(\R[H])^{\times}  \ar@{->}[d]^{\nr_{\R[H]}} \ar@{->}[rrrr]^{\textrm{inclusion}} & & & & (\R[G])^{\times} \ar@{->}[d]_{\nr_{\R[G]}} \\
\zeta(\R[H])^{\times +} \ar@{->}[rrrr]^{i^{G}_{H}} \ar@{->}[d]^{\onr_{\R[H]}^{-1}}_{\simeq}
\ar@/_2.5pc/[dd]_{\partial_{\Z[H], \R[H]}}
& & & & \zeta(\R[G])^{\times +}  \ar@/^2.5pc/[dd]^{\partial_{\Z[G], \R[G]}}
\ar@{->}[d]_{\onr_{\R[G]}^{-1}}^{\simeq} \\
K_{1}(\R[H])  \ar@{->}[d] \ar@{->}[rrrr]^{\mathrm{ind}^{G}_{H}} & & & & K_{1}(\R[G])  \ar@{->}[d] \\
K_{0}(\Z[H], \R[H]) \ar@{->}[rrrr]^{\mathrm{ind}^{G}_{H}} & & & & K_{0}(\Z[G], \R[G]).
}
\end{equation}

\subsection{The extended boundary homomorphism}
We recall some properties of the `extended boundary homomorphism'
$\hat{\partial}_{\Z[G], \R[G]} : \zeta(\R[G])^{\times} \rightarrow K_{0}(\Z[G], \R[G])$
first introduced in \cite[Lemma 9]{MR1884523} (a more conceptual description
is given in \cite[Lemma 2.2]{MR2371375}). The restriction of $\hat{\partial}_{\Z[G], \R[G]}$
to $\zeta(\R[G])^{\times +}$ is $\partial_{\Z[G], \R[G]}$.

\begin{lemma}\label{lemma:commutes-up-to-2}
Letting $\alpha$ and $\beta$ denote the natural inclusions, the diagram
\[
\xymatrix@1@!0@=36pt {
\zeta(\R[G])^{\times}  \ar@{->}[d]^{\alpha} \ar@{->}[rrrr]^{\hat{\partial}_{\Z[G], \R[G]}}
& & & & K_{0}(\Z[G], \R[G]) \ar@{->}[d]^{\beta} & & & \\
\zeta(\C[G])^{\times} \ar@{->}[rrrr]^{\partial_{\Z[G], \C[G]}}
& & & & K_{0}(\Z[G], \C[G])
}
\]
commutes up to elements of order $2$.
In other words, given $x \in \zeta(\R[G])^{\times}$ we have
\[
\beta(\hat{\partial}_{\Z[G], \R[G]}(x)) = \partial_{\Z[G], \C[G]}(\alpha(x)) \cdot u
\]
for some $u \in K_{0}(\Z[G], \C[G])$ of order at most $2$.
\end{lemma}

\begin{proof}
Reduced norms commute with extension of scalars, and squares in $\zeta(\R[G])^{\times}$
are reduced norms. Thus for $x \in \zeta(\R[G])^{\times}$ we have
\[
\beta(\hat{\partial}_{\Z[G], \R[G]}(x))^{2}
= \beta(\hat{\partial}_{\Z[G], \R[G]}(x^{2}))
= \beta(\partial_{\Z[G], \R[G]}(x^{2}))
= \partial_{\Z[G], \C[G]}(\alpha(x^{2}))
= \partial_{\Z[G], \C[G]}(\alpha(x))^{2},
\]
from which the desired result follows immediately.
\end{proof}

Regarding $G$ as fixed, we henceforth abbreviate $\hat{\partial}_{\Z[G], \R[G]}$ and
$\partial_{\Z[G], \R[G]}$ to $\hat{\partial}$ and $\partial$, respectively.

\section{Centres of complex group algebras}\label{sec:complex-group-alg}

Let $G$ be a finite group and let $\mathrm{Irr}(G)$ be the set of irreducible complex characters of $G$.
Recall that there is a canonical isomorphism $\zeta(\C[G]) = \prod_{\chi \in \mathrm{Irr}(G)} \C$.
We shall henceforth use this identification without further mention.

\subsection{Explicit induction}
Let $H$ be a subgroup of $G$ and define a map
\begin{equation}\label{eqn:explicit-induction}
i_{H}^{G}: \zeta(\C[H]) \rightarrow \zeta(\C[G]), \quad
(\alpha_{\psi})_{\psi \in \mathrm{Irr}(H)} \mapsto
\left( \prod_{\psi \in \mathrm{Irr}(H)}
\alpha_{\psi}^{\langle \chi |_{H}, \psi \rangle_{H}} \right)_{\chi \in \mathrm{Irr}(G)},
\end{equation}
where $\langle \chi |_{H}, \psi \rangle_{H}$ denotes the usual inner
product of characters of $H$. The restriction of this map to
$\zeta(\C[H])^{\times}$ is the same as the map
$i_{H}^{G}: \zeta(\C[H])^{\times} \rightarrow \zeta(\C[G])^{\times}$
defined in \S \ref{subsec:induction} (with $F=\C$) so that using
the same name for these maps is justified. This map restricts further to
$i_{H}^{G}: \zeta(\R[H])^{\times +} \rightarrow \zeta(\R[G])^{\times +}$
(as defined in \S \ref{subsec:induction} with $F=\R$).

\subsection{The involution \#}
We write $\alpha \mapsto \alpha^{\#}$ for the involution of $\zeta(\C[G])$ induced by the
$\C$-linear anti-involution of $\C[G]$ that sends each element of $G$ to its inverse.
If $\alpha = (\alpha_{\chi})_{\chi \in \mathrm{Irr}(G)}$
then $\alpha^{\#} = (\alpha_{\bar{\chi}})_{\chi \in \mathrm{Irr}(G)}$.
Furthermore, $\#$ restricts to an involution of $\zeta(\R[G])^{\times+}$ which is compatible with induction, i.e., if $\alpha \in \zeta(\R[H])^{\times +}$ then $i_{H}^{G}(x^{\#}) = i_{H}^{G}(x)^{\#}$.

\subsection{Meromorphic $\zeta(\C[G])$-valued functions}
A meromorphic $\zeta(\C[G])$-valued function is a function of a complex variable $s$
of the form $s \mapsto g(s) = (g(s, \chi))_{\chi \in \mathrm{Irr}(G)}$ where each function
$s \mapsto g(s, \chi)$ is meromorphic. If $r(\chi)$ denotes the order of vanishing
of $g(s, \chi)$ at $s=0$ then we set
$g^{*}(0,\chi) := \lim_{s \rightarrow 0} s^{-r(\chi)} g(s,\chi)$ and
$g^{*}(0) := (g^{*}(0,\chi))_{\chi \in \mathrm{Irr}(G)} \in \zeta(\C[G])^{\times}$.

\section{$L$-functions}\label{sec:L-functs}

Let $K/k$ be a finite Galois extension of number fields with Galois group $G$
and let $\mathcal{S}$ be a finite set of places of $k$ containing the infinite places
$\mathcal{S}_{\infty}$.

\subsection{Artin $L$-functions}
Let $\frakp$ be a finite prime of $k$. Let $\psi$ be a complex character of $G_{\frakp}$
and choose a $\C[G_{\frakp}]$-module $V_{\psi}$ with character $\psi$. Recalling the
notation from \S \ref{sec:U-chi} we define
\begin{equation}\label{eqn:def-local-L}
L_{K_{\Frakp}/k_{\frakp}}(s, \psi) := \det_{\C}(1-F_{\frakp}(\mathrm{N}\frakp)^{-s} \mid V_{\psi}^{G_{0, \frakp}})^{-1},
\end{equation}
where $\mathrm{N}\frakp$ is the cardinality of the residue field of $\frakp$.
Note that $L_{K_{\Frakp}/k_{\frakp}}(s, \psi)$ only depends on $\frakp$ and not on the choice
of $\Frakp$. Furthermore, it is easy to see that
\[
L_{K_{\Frakp}/k_{\frakp}}(s, \psi+\psi')
=L_{K_{\Frakp}/k_{\frakp}}(s, \psi)L_{K_{\Frakp}/k_{\frakp}}(s, \psi')
\]
for two characters $\psi,\psi'$ of $G_{\frakp}$; thus the definition extends to all virtual
characters of $G_{\frakp}$.

Now let $\chi \in \mathrm{Irr}(G)$ and for each $\frakp$ let $\chi_{\frakp}$ denote the restriction of $\chi$ to $G_{\frakp}$. The Artin $L$-function attached to $\mathcal{S}$ and $\chi$
is defined as an infinite product
\begin{equation}\label{eqn:prod-local-L}
L_{K/k, \mathcal{S}}(s,\chi)
:= \prod_{\frakp \notin \mathcal{S}} L_{K_{\Frakp}/k_{\frakp}}(s, \chi_{\frakp})
\end{equation}
which converges for $\mathrm{Re}(s)>1$ and can be extended to the whole
complex plane by meromorphic continuation.

\subsection{Equivariant $L$-functions}
We define meromorphic $\zeta(\C[G_{\frakp}])$-valued functions by
\[
L_{K_{\Frakp}/k_{\frakp}}(s) := (L_{K_{\Frakp}/k_{\frakp}}(s, \psi))_{\psi \in \mathrm{Irr}(G_{\frakp})}
\]
and define the equivariant Artin $L$-function to be the meromorphic $\zeta(\C[G])$-valued function
\[
L_{K/k, \mathcal{S}}(s)
:= (L_{K/k, \mathcal{S}}(s,\chi))_{\chi \in \mathrm{Irr}(G)}.
\]
From \eqref{eqn:explicit-induction} and \eqref{eqn:prod-local-L} it is straightforward to check
that for $\mathrm{Re}(s)>1$ we have
\begin{equation}\label{eqn:L-func-induc-form}
L_{K/k, \mathcal{S}}(s)
= \prod_{\frakp \notin \mathcal{S}} i_{G_{\frakp}}^{G}(L_{K_{\Frakp}/k_{\frakp}}(s)).
\end{equation}
Note that $L^{*}_{K_{\Frakp}/k_{\frakp}}(0) \in \zeta(\R[G_{\frakp}])^{\times +}$
and $L^{*}_{K/k, \mathcal{S}}(0) \in \zeta(\R[G])^{\times}$ (see \cite[Lemma 2.7]{MR2371375}).
We henceforth abbreviate $L_{K/k, \mathcal{S}}(s)$ to $L_{\mathcal{S}}(s)$ and
$L_{\mathcal{S_{\infty}}}(s)$ to $L(s)$.

\section{Tate sequences, refined Euler characteristics and the ETNC}\label{sec:tate-ETNC}

Let $K/k$ be a finite Galois extension of number fields with Galois group $G$.
By $S$ we denote a finite $G$-stable set of places of $K$ containing the set of archimedean places $S_{\infty}$. We let $S_{\mathrm{ram}}$ denote the places of $K$ ramified in $K/k$.
The set of $k$-places below places in $S$ (respectively, $S_{\infty}$, $S_{\mathrm{ram}}$)
will be written $\mathcal{S}$ (respectively, $\mathcal{S}_{\infty}$, $\mathcal{S}_{\mathrm{ram}}$).
(Note that this is different from the notation used in \cite{MR2371374}.)
Let $E_{S} = \mathcal{O}_{K,S}^{\times}$ and $\Delta S$ be the kernel of the augmentation
map $\Z S \rightarrow \Z$. We shall henceforth abbreviate `cohomologically trivial' to `c.t.'\
and `finitely generated' to `f.g.' Note that as $G$ is finite, `$G$-c.t.'\ is equivalent to `of projective dimension at most $1$ over $\Z[G]$'.

Now let $S'$ denote a finite $G$-stable set of places of $K$ that is `large', i.e.,
$S_{\infty} \cup S_{\mathrm{ram}} \subseteq S'$ and $\cl_{K,S'}=0$.
Tate defined a canonical class
$\tau = \tau_{S'} \in \mathrm{Ext}_{\Z[G]}^{2}(\Delta S', E_{S'})$
(see \cite{MR0207680}, \cite[Chapter II]{MR782485}). The fundamental properties
of $\tau$ ensure the existence of so-called Tate sequences, that is,
four term exact sequences of f.g.\ $\Z[G]$-modules
\begin{equation}\label{eqn:tate-seq}
0 \longrightarrow E_{S'} \longrightarrow A \longrightarrow B \longrightarrow \Delta S' \longrightarrow 0
\end{equation}
representing $\tau$ with $A$ $G$-c.t.\ and $B$ projective. In \cite{MR1394524},
Ritter and Weiss construct a Tate sequence for $S$ `small'
\[
0 \longrightarrow E_{S} \longrightarrow A \longrightarrow B \longrightarrow \nabla \longrightarrow 0
\]
where $\nabla$ is given by a short exact sequence
$0 \rightarrow \cl_{K,S} \rightarrow \nabla \rightarrow \overline{\nabla} \rightarrow 0$.

In order to take full advantage of these sequences we shall use refined Euler characteristics,
which we now briefly review in the special case of interest to us.
For a full account, we refer the reader to \cite{MR2076565} (also see \cite[\S 1.2]{MR1863302}).
Following \cite[\S 3]{MR2371374}, we adopt a slight change from the usual convention (this only
results in a sign change).
A `metrised' complex over $\Z[G]$ consists of a complex in degrees $0$ and $1$
\[
A \longrightarrow B,
\]
together with an $\R[G]$-isomorphism
\begin{equation}\label{RVRU}
\varphi: \R \otimes U \longrightarrow \R \otimes V
\end{equation}
where both $A$ and $B$ are f.g.\ and c.t.\ over $G$, and $U$ (respectively $V$) is the kernel
(respectively cokernel) of $A \rightarrow B$.
To every metrised complex $E=(A \rightarrow B, \varphi)$
we can associate a refined Euler characteristic $\euler(E) \in K_{0}(\Z[G], \R[G])$
as follows. We can write down a four-term exact sequence
\begin{equation}\label{UABV}
0 \longrightarrow U \longrightarrow A \longrightarrow B
\longrightarrow V \longrightarrow 0,
\end{equation}
which gives rise to the tautological exact sequences
\begin{eqnarray*}
0 \longrightarrow \ker(\R \otimes B \rightarrow \R \otimes V) \longrightarrow \R \otimes B \longrightarrow \R \otimes V \longrightarrow 0,\\
0 \longrightarrow \R \otimes U \longrightarrow \R \otimes A \longrightarrow  \im(\R \otimes A \rightarrow \R \otimes B) \longrightarrow 0.
\end{eqnarray*}
We choose splittings for these sequences and obtain an isomorphism
$\tilde{\varphi}: \R \otimes A \rightarrow \R \otimes B$
\begin{eqnarray*}
\R \otimes A &\cong&
\im(\R \otimes A \rightarrow \R \otimes B)  \oplus (\R \otimes U)
= \ker(\R \otimes B \rightarrow \R \otimes V) \oplus (\R \otimes U) \\
&\cong& \ker(\R \otimes B \rightarrow \R \otimes V) \oplus (\R \otimes V) \\
&\cong& \R \otimes B
\end{eqnarray*}
where the first and third maps are obtained by the chosen splittings and the second
map is induced by $\varphi$. (We refer to $\tilde{\varphi}$ as a `transpose' of $\varphi$.)
If $A$ and $B$ are both $\Z[G]$-projective, we define
$\euler(A \rightarrow B, \varphi) = (A, \tilde{\varphi}, B) \in K_{0}(\Z[G], \R[G])$. This definition
can be extended to the more general case where $A$ and $B$ are c.t.\ over $G$.
In all cases, $\euler(A \rightarrow B, \varphi)$ can be shown to be independent of the choice
of splittings.

We note several properties of $\euler$. Firstly,
$\euler(A \rightarrow B, \varphi)$ remains unchanged if $\varphi$ is composed
with an automorphism of determinant $1$ on either side (see \cite[Proposition 1.2.1(ii)]{MR1863302}).
Secondly, if the metrisation (\ref{RVRU}) is given, then the class of the exact sequence
(\ref{UABV}) in $\mathrm{Ext}_{\Z[G]}^{2}(V,U)$ uniquely determines
$\euler(A \rightarrow B, \varphi)$ (see \cite[Proposition 1.2.2 and Remark 1.2.3]{MR1863302}). Finally,
it is straightforward to show that $\euler$ is compatible with induction, i.e.,
if $H$ is a subgroup of $G$ and $A \rightarrow B$ is an appropriate complex of $\Z[H]$-modules
with metrisation $\varphi$, then
\begin{equation}\label{eqn:euler-induction}
\mathrm{ind}_{H}^{G}(\euler(A \rightarrow B, \varphi)) = \euler(\mathrm{ind}_{H}^{G} A \rightarrow \mathrm{ind}_{H}^{G} B, \mathrm{ind}_{H}^{G} \varphi).
\end{equation}

Now let $E$ be the complex formed by the middle two terms of the Tate sequence
(\ref{eqn:tate-seq}) and metrise it by setting
$U=E_{S'}$, $V=\Delta S'$ and $\varphi^{-1}:\R E_{S'} \rightarrow \R \Delta S'$
the negative of the usual Dirichlet map, so $\varphi^{-1}(u) = - \sum_{v \in S'} \log |u|_{v} \cdot v$.
The equivariant Tamagawa number is defined to be
\[
T\Omega(K/k, 0) := \psi^{*}_{G}(\hat{\partial}(L_{\mathcal{S}'}^{*}(0)^{\#})-\euler(E)) \in K_{0}(\Z[G], \R[G]),
\]
where $\psi^{*}_{G}$ is a certain involution of $K_{0}(\Z[G], \R[G])$ which can be ignored for
our purposes. (Note that we have $-\euler(E)$ rather than $+\euler(E)$ here because, as mentioned above, our definition of $\euler$ is slightly different from the usual convention, resulting in a sign change.) The equivariant Tamagawa number conjecture (ETNC) in this context (i.e. for the motive $h^{0}(K)$ with coefficients in $\Z[G]$) simply
states that $T\Omega(K/k, 0)$ is zero (see \cite{MR1884523}, \cite{MR1863302}).
One can also re-interpret other well-known conjectures
using this framework. For example, Stark's Main Conjecture is equivalent to the
statement that $T\Omega(K/k, 0)$ belongs to $K_{0}(\Z[G], \Q)$.
If $\mathcal{M}$ is a maximal $\Z$-order
in $\Q[G]$ containing $\Z[G]$ then the Strong Stark Conjecture
can be interpreted as
\[
T\Omega(K/k, 0) \in K_{0}(\Z[G], \Q[G])_{\mathrm{tors}}
= \ker(K_{0}(\Z[G], \Q[G]) \rightarrow K_{0}(\mathcal{M}, \Q[G])),
\]
i.e., `the ETNC holds modulo torsion' (see \cite[\S 2.2]{MR1863302}).

\section{Metrised commutative diagrams}\label{sec:metrised-diagrams}

We briefly review Greither's construction of certain metrised commutative diagrams
which we shall use in a crucial way;
we leave the reader to consult \cite{MR2371374} for further details.
The principal tool used in this construction is the aforementioned
`Tate sequence for small $S$' of Ritter and Weiss (see \cite{MR1394524}).
Although the only application in \cite{MR2371374} is in the case that $G$ is abelian,
the explicit construction given there is in fact valid for the general case once one makes
the minor modifications to `the core diagram' of \cite[\S 5]{MR2371374} described in the
proof of \cite[Proposition 4.4]{nickel-LRNC}.

We adopt the setup and notation of \S \ref{sec:tate-ETNC} and add the further
hypotheses that $k$ is totally real, $K$ is a CM field and $S'$ is `larger' in the sense
of \cite{MR1394524}, that is: $S_{\infty} \cup S_{\mathrm{ram}} \subseteq S'$,
$\cl_{K,S'}=0$, and $G=\cup_{\frakp \in S'} G_{\frakp}$.
We write $j$ for the unique complex conjugation in $G$ and define
$R := \Z[G][\frac{1}{2}]/(1+j)$. For every $G$-module $M$ we let
$M^{-} := R \otimes_{\Z[G]} M$ (this notation, which includes inversion of $2$, is used in \cite{MR2371374} and is non-standard but practical).
Note that the construction of refined Euler characteristic also works for complexes over $R$.

Let $C$ be the free
$\Z[G]$-module with basis elements $x_{\frakp}$, where $\frakp$ runs over
$\mathcal{S}' \backslash \mathcal{S}_{\infty}$. Using the Tate sequences for $S$ `larger' and
for $S$ `small', Greither constructs the following diagrams.

\begin{equation}
\tag{D1}
\xymatrix@1@!0@=36pt {
E_{S'} \ar@{->}[rr] & & A \ar@{->}[rr] & & B_{S'} \ar@{->}[rr] & & \Delta S' \\
C \oplus E_{S_{\infty}} \ar@{->}[rr] \ar@{->}[u] & & C \oplus A \ar@{->}[rr] \ar@{->}[u] & &
B_{S_{\infty}} \ar@{->}[rr] \ar@{->}[u] & & \nabla \ar@{->}[u] \\
Z' \ar@{->}[rr] \ar@{->}[u] & & C \ar@{->}[rr] \ar@{->}[u] & & C \ar@{->}[rr] \ar@{->}[u] & & Z'' \ar@{->}[u] \\
}
\end{equation}

\begin{equation}
\tag{D2}
\xymatrix@1@!0@=36pt {
E_{S_{\infty}} \ar@{->}[rr] & & A \ar@{->}[rr] & & \tilde{B} \ar@{->}[rr] & & \nabla / \delta(C) \\
C \oplus E_{S_{\infty}} \ar@{->}[rr] \ar@{->}[u] & & C \oplus A \ar@{->}[rr] \ar@{->}[u] & &
B_{S_{\infty}} \ar@{->}[rr] \ar@{->}[u] & & \nabla \ar@{->}[u] \\
C \ar@{->}[rr]^{\mathrm{id}} \ar@{->}[u] & & C \ar@{->}[rr]^{0} \ar@{->}[u] & &
C \ar@{->}[rr]^{\mathrm{id}} \ar@{->}[u] & & C \ar@{->}[u]^{\delta} \\
}
\end{equation}

In the original construction $S$ is used in place of $S_{\infty}$,
and only later does the author specialise to the case $S=S_{\infty}$. Note that the middle
rows of (D1) and (D2) are identical and that the middle map of the bottom row of (D1) is
far from the identity in general. The `minus part' of each diagram is denoted (D1)$^{-}$ or
(D2)$^{-}$, as appropriate.

In \cite[\S 7]{MR2371374} each row is given a metrisation and we label the corresponding
refined Euler characteristics as follows:
\begin{eqnarray*}
X_{S'} &=& \euler(\textrm{top row of (D1)}),\\
X_{1} &=& \euler(\textrm{middle row of (D1)}) = \euler(\textrm{middle row of (D2)}),\\
X_{C} &=& \euler(\textrm{bottom row of (D1)}),\\
X_{\infty}^{-} &=& \euler(\textrm{top row of (D2)}^{-}),\\
X_{2} &=& \euler(\textrm{bottom row of (D2)}).
\end{eqnarray*}
(Note that there is a typo in the definition of $X_{\infty}^{-}$ in \cite[top of p.1418]{MR2371374}.)
The metrisations are chosen to be `compatible' within the minus part of each diagram
(see  \cite[Lemmas 7.3 and 7.4]{MR2371374}) so that we have
\[
X_{1}^{-} = X_{S'}^{-} + X_{C}^{-} \quad \textrm{ and } \quad
X_{1}^{-} = X_{\infty}^{-} + X_{2}^{-} \quad \textrm{ in } \quad K_{0}(R, \R[G]^{-}),
\]
where we denote the natural map $K_{0}(\Z[G], \R[G]) \rightarrow K_{0}(R, \R[G]^{-})$ by a minus exponent.
Putting these two equations together gives the following.

\begin{prop}\label{prop:X-eqn}
We have $X_{\infty}^{-} = X^{-}_{S'} + X_{C}^{-} - X_{2}^{-}$ in $K_{0}(R, \R[G]^{-})$.
\end{prop}

\section{Computing refined Euler characteristics}

We compute the refined Euler characteristics of the metrised commutative diagrams of
\S \ref{sec:metrised-diagrams} (i.e. of \cite[\S 3]{MR2371374}) in the non-abelian case.
Recall the definitions of $e_{\frakp}', e_{\frakp}'', \bar{e}_{\frakp}, \bar{\bar{e}}_{\frakp}$
from \S \ref{sec:U-chi} and note that even though $\bar{\bar{e}}_{\frakp}e_{\frakp}'' = e_{\frakp}''$, we sometimes retain $\bar{\bar{e}}_{\frakp}$ for clarity.

Recall that $h$ is a positive integer multiple of $|\cl_{K}|$ (see \cite[p.1411]{MR2371374}).

\begin{lemma}\label{lem:vp}
Let $v_{\frakp} = h |G_{\frakp}| \cdot \bar{e}_{\frakp} + \bar{\bar{e}}_{\frakp} \in \Q[G_{\frakp}]$.
Then in $ K_{0}(\Z[G], \R[G])$ we have
\[
X_{C} = \sum_{\frakp \in \mathcal{S}' \backslash \mathcal{S}_{\infty}} \partial (\nr_{\R[G]}(v_{\frakp})).
\]
\end{lemma}

\begin{proof}
This is proven in the abelian case in \cite[Lemma 7.6]{MR2371374}.
The key point is that the `transposed isomorphism' at the local level
is multiplication by $v_{\frakp}$ (note this is central in $\R[G_{\frakp}]$),
and this part of the proof holds without change in the non-abelian case.
Hence
\[
X_{C} =
\sum_{\frakp \in \mathcal{S}' \backslash \mathcal{S}_{\infty}}
\ind ( (\Z[G_{\frakp}], (v_{\frakp})_{r}, \Z[G_{\frakp}]) ) =
\sum_{\frakp \in \mathcal{S}' \backslash \mathcal{S}_{\infty}}
(\Z[G], (v_{\frakp})_{r}, \Z[G]) =
\sum_{\frakp \in \mathcal{S}' \backslash \mathcal{S}_{\infty}} \partial (\nr_{\R[G]}(v_{\frakp})),
\]
where the equalities are due, respectively, to the compatibility of $\euler$ with induction
(i.e. equation \eqref{eqn:euler-induction} with $H=G_{\frakp}$), the commutativity of diagram
\eqref{eqn:K-diagram} (again with $H=G_{\frakp}$), and the explicit formula
\eqref{eqn:explicit-partial-hom} (with $R=\Z$ and $F=\R$).
\end{proof}

Recall that $h_{\frakp} = g_{\frakp} \cdot e_{\frakp}' + e''_{\frakp}$ where
$g_{\frakp} = |G_{0,\frakp}|+1-F_{\frakp}^{-1}$ (see \cite[p.1420]{MR2371374}).

\begin{lemma}\label{tp}
Let
\[
t_{\frakp}
= h \log \mathrm{N} \Frakp \cdot \bar{e}_{\frakp}
+ \frac{1-F_{\frakp}^{-1}}{h_{\frakp}} \cdot \bar{\bar{e}}_{\frakp} e_{\frakp}'
+ \bar{\bar{e}}_{\frakp} e_{\frakp}''.
\]
Then in $K_{0}(R, \R[G]^{-})$ we have
\[
X_{2}^{-} =  \sum_{\frakp \in \mathcal{S}' \backslash \mathcal{S}_{\infty}} \partial (\nr_{\R[G]}(t_{\frakp}))^{-}.
\]
\end{lemma}

\begin{proof}
Recall that the bottom row of (D2) is metrised with the map $\psi:\R C \rightarrow \R C$
(see \cite[p.1417]{MR2371374}) and that $X_{2}$ is the associated refined Euler characteristic
(see \S \ref{sec:metrised-diagrams}). Because of the zero map in the middle of the row, the transpose of $\psi$ is just
$\psi$ itself. In \cite[Lemma 8.6]{MR2371374} it is shown that in the minus part
$\psi$ is the direct sum of maps induced from endomorphisms
$\psi_{\frakp}$ of $\R[G_{\frakp}] \cdot x_{\frakp}$ given by multiplication by $t_{\frakp}$
(note this is central in $\R[G_{\frakp}]$), and this holds without change in the
non-abelian case.
The result then follows in the same way as in the proof of Lemma \ref{lem:vp} above,
but with $v_{\frakp}$ replaced by $t_{\frakp}$.
\end{proof}

\begin{definition}
Fix a finite prime $\frakp$ of $k$. Let $\psi \in \mathrm{Irr}(G_{\frakp})$ and let $e_{\psi}$ be the primitive central idempotent of $\C[G_{\frakp}]$ attached to $\psi$.
In the spirit of \cite[p.1421]{MR2371374}, we say that
\begin{itemize}
\item[1.] $\psi \in T_{1}(\frakp)$ if $\psi$ is trivial, i.e., $e_{\psi}=\bar{e}_{\frakp}$;
\item[2.] $\psi \in T_{2}(\frakp)$ if $\psi$ is non-trivial but trivial on $G_{0,\frakp}$, i.e.,
$e_{\psi}\bar{e}_{\frakp}=0$ but $e_{\psi} e_{\frakp}'=e_{\psi}$;
\item[3.] $\psi \in T_{3}(\frakp)$ if $\psi$ is non-trivial on $G_{0,\frakp}$, i.e., $e_{\psi} e_{\frakp}'=0$.
\end{itemize}
This division into types corresponds to the decomposition of $1$ into orthogonal idempotents
\[
1 = \bar{e}_{\frakp} + \bar{\bar{e}}_{\frakp} e_{\frakp}' + \bar{\bar{e}}_{\frakp} e_{\frakp}''
\]
in $\Q[G_{\frakp}] \subseteq \C[G_{\frakp}]$, where $\psi \in T_{i}(\frakp)$ corresponds to
$e_{\psi}$ sending the $i$th of the right-hand summands to $e_{\psi}$ and the other two
to $0$, for $i=1,2,3$. Note that if $\psi \in T_{2}(\frakp)$ then $\psi$ factors through
the cyclic group $\overline{G}_{\frakp} = G_{\frakp} / G_{0, \frakp}$ and so $\psi$ is linear.
\end{definition}

\begin{lemma}\label{chi-parts}
Fix a prime $\frakp \in \mathcal{S}' \backslash \mathcal{S}_{\infty}$
and let $\psi \in \mathrm{Irr}(G_{\frakp})$. We have
\[
e_{\psi}(v_{\frakp} t_{\frakp}^{-1}h_{\frakp}^{-1})
= \left\{
\begin{array}{ll}
(e_{\psi} (\log \mathrm{N} \frakp))^{-1} & \textrm{ if } \psi \in T_{1}(\frakp);\\
(e_{\psi} (1-F_{\frakp}^{-1}))^{-1} & \textrm{ if } \psi \in T_{2}(\frakp);\\
e_{\psi} & \textrm{ if } \psi \in T_{3}(\frakp).
\end{array}
\right.
\]
\end{lemma}

\begin{proof}
Suppose $\psi \in T_{1}(\frakp)$. Then
$e_{\psi}v_{\frakp} = e_{\psi}h|G_{\frakp}|$, $e_{\psi}t_{\frakp} = e_{\psi}h \log \mathrm{N} \Frakp$
and $e_{\psi}h_{\frakp} = e_{\psi} |G_{0,\frakp}|$. Hence
$e_{\psi}(v_{\frakp} t_{\frakp}^{-1}h_{\frakp}^{-1})
= e_{\psi} [G_{\frakp}:G_{0,\frakp}](\log \mathrm{N} \Frakp)^{-1}
= (e_{\psi}(\log \mathrm{N} \frakp))^{-1}$.
Suppose $\psi \in T_{2}(\frakp)$. Then
$e_{\psi}v_{\frakp} = e_{\psi}$ and $e_{\psi}t_{\frakp} = e_{\psi}h_{\frakp}^{-1}(1-F_{\frakp}^{-1})$. Hence $e_{\psi}(v_{\frakp} t_{\frakp}^{-1}h_{\frakp}^{-1}) = (e_{\psi} (1-F_{\frakp}^{-1}))^{-1}$.
If $\psi \in T_{3}(\frakp)$ then $e_{\psi}v_{\frakp}=e_{\psi}t_{\frakp}=e_{\psi}h_{\frakp} = e_{\psi}$
and so $e_{\psi}(v_{\frakp} t_{\frakp}^{-1}h_{\frakp}^{-1}) = e_{\psi}$.
\end{proof}

\begin{lemma}\label{L-func-local}
We have
$L_{\mathcal{S}'}^{*}(0)^{\#} \nr_{\R[G]} (\prod_{\frakp \in \mathcal{S}' \backslash \mathcal{S}_{\infty}}
v_{\frakp} t_{\frakp}^{-1} h_{\frakp}^{-1})= L^{*}(0)^{\#}$ in
$\zeta(\R[G])^{\times}$.
\end{lemma}

\begin{proof}
First recall the definitions from \S \ref{sec:complex-group-alg} and \S \ref{sec:L-functs}.
From \eqref{eqn:L-func-induc-form}, we have
\[
L^{*}(0)^{\#}(L_{\mathcal{S}'}^{*}(0)^{\#})^{-1}
= (L^{*}(0)L_{\mathcal{S}'}^{*}(0)^{-1})^{\#}
= \prod_{\frakp \in \mathcal{S}' \backslash \mathcal{S}_{\infty}}
i_{G_{\frakp}}^{G}(L_{K_{\Frakp}/k_{\frakp}}^{*}(0))^{\#}
\quad \textrm{in} \quad \zeta(\R[G])^{\times}.
\]
Since $\#$ is compatible with induction, we are hence reduced to showing that
\[
\prod_{\frakp \in \mathcal{S}' \backslash \mathcal{S}_{\infty}} \nr_{\R[G]}
(v_{\frakp} t_{\frakp}^{-1} h_{\frakp}^{-1} )
= \prod_{\frakp \in \mathcal{S}' \backslash \mathcal{S}_{\infty}}
i_{G_{\frakp}}^{G}(L_{K_{\Frakp}/k_{\frakp}}^{*}(0)^{\#})
\quad \textrm{in} \quad \zeta(\R[G])^{\times +}.
\]
Fix $\frakp \in \mathcal{S}' \backslash \mathcal{S}_{\infty}$.
We have $\nr_{\R[G]}( v_{\frakp} t_{\frakp}^{-1}h_{\frakp})
= i_{G_{\frakp}}^{G}(\nr_{\R[G_{\frakp}]}( v_{\frakp} t_{\frakp}^{-1}h_{\frakp}))$
from the commutative diagram \eqref{eqn:K-diagram} (with $H=G_{\frakp}$).
Moreover, Lemma \ref{chi-parts} shows that
$\nr_{\R[G_{\frakp}]}(v_{\frakp} t_{\frakp}^{-1}h_{\frakp}^{-1}) = v_{\frakp} t_{\frakp}^{-1}h_{\frakp}^{-1}$ since $\psi$ is linear if $\psi \in T_{1}(\frakp) \cup T_{2}(\frakp)$. Therefore
we are further reduced to verifying that
\[
v_{\frakp} t_{\frakp}^{-1} h_{\frakp}^{-1}
= L_{K_{\Frakp}/k_{\frakp}}^{*}(0)^{\#} \quad \textrm{in} \quad \zeta(\R[G_{\frakp}])^{\times +}.
\]
However, this holds because for each $\psi \in \mathrm{Irr}(G_{\frakp})$ we have
\[
e_{\psi} L_{K_{\Frakp}/k_{\frakp}}^{*}(0)^{\#}
= e_{\psi} L_{K_{\Frakp}/k_{\frakp}}^{*}(\bar{\psi}, 0)
= e_{\psi}(v_{\frakp} t_{\frakp}^{-1}h_{\frakp}^{-1}),
\]
where the second equality follows from Lemma \ref{chi-parts}
and a direct computation using the definition of local $L$-function \eqref{eqn:def-local-L}.
\end{proof}

\begin{definition}
We say that a character $\chi$ of $G$ is \emph{odd} if $j$
(the unique complex conjugation in $G$) acts as $-1$ on a $\C[G]$-module $V_{\chi}$
with character $\chi$ or, equivalently, $e_{\chi}e_{-} = e_{\chi}$ in $\C[G]$ where
$e_{-} := \frac{1-j}{2}$.
\end{definition}

\begin{prop}\label{prop:explicit-eqn}
Assume that ETNC holds for the motive $h^{0}(K)$ with coefficients in $R$.
Let $\hg := \prod_{\frakp \in \mathcal{S}' \backslash \mathcal{S}_{\infty}} h_{\frakp}$
(as in  \cite[\S 8]{MR2371374}).
Then $\hat{\partial}(L(0)^{\#} \nr_{\R[G]}(\hg))^{-} = X_{\infty}^{-}$ in $K_{0}(R, \R[G]^{-})$.
\end{prop}

\begin{proof}
The ETNC for $h^{0}(K)$ with coefficients in $\Z[G]$ gives
$\hat{\partial}(L_{\mathcal{S}'}^{*}(0)^{\#}) = X_{S'}$
(recall the exposition in \S \ref{sec:tate-ETNC} and note that $X_{S'}=\euler(E)$).
Hence the ETNC for $h^{0}(K)$ with coefficients in $R$ gives
$\hat{\partial}(L_{\mathcal{S}'}^{*}(0)^{\#})^{-} = X_{S'}^{-}$.
Let
\[
f := L_{\mathcal{S}'}^{*}(0)^{\#}
\textstyle{\prod_{\frakp \in \mathcal{S}' \backslash \mathcal{S}_{\infty}}}
\nr_{\R[G]}(v_{\frakp} t_{\frakp}^{-1}) \in \zeta(\R[G])^{\times}.
\]
Then combining Lemmas \ref{lem:vp} and \ref{tp} with Proposition \ref{prop:X-eqn}
gives $\hat{\partial}(f)^{-}=X_{\infty}^{-}$. However, from Lemma \ref{L-func-local} we deduce that
\[
f=L^{*}(0)^{\#} \nr_{\R[G]}
( \textstyle{\prod_{\frakp \in \mathcal{S}' \backslash \mathcal{S}_{\infty}}} v_{\frakp}^{-1} t_{\frakp} h_{\frakp})
\textstyle{\prod_{\frakp \in \mathcal{S}' \backslash \mathcal{S}_{\infty}}} \nr_{\R[G]}(v_{\frakp} t_{\frakp}^{-1})
= L^{*}(0)^{\#} \nr_{\R[G]}(\hg)
\]
in the minus part. A standard argument shows that $L(0,\chi) = L^{*}(0,\chi)$ for
every odd irreducible character $\chi$ of $G$ (this is a straightforward exercise
once one has the order of vanishing formula \eqref{eqn:order-of-vanishing} used in
\S \ref{sec:main-results}) and so we have $L^{*}(0)^{-}=L(0)^{-}$.
Therefore $f=L(0)^{\#} \nr_{\R[G]}(\hg)$ in the minus part and
the desired result now follows by applying $\hat{\partial}$ to both sides.
\end{proof}

\section{Computing Fitting ideals}

\begin{prop}\label{prop:fit-comp}
Let $\chi \in \mathrm{Irr}(G)$ with $\chi$ odd.
Then ignoring $2$-parts, we have
\[
\nr_{e_{\chi}E[G]}(e_{\chi} \hgi) U_{\chi}
\frac{\fit(H^{2}(G, \mu_{K}[\chi]))}{\fit(H^{1}(G, \mu_{K}[\chi]))} \fit( (\nabla/\delta(C))[\chi]_{G}) \subseteq \fit(\cl_{K}[\chi]^{G}).
\]
\end{prop}

\begin{remark}
If $\mu_{K}^{-}$ is $R$-c.t.\ then $(\nabla/\delta(C))^{-}$ is also $R$-c.t.\ and
our argument shows that
\[
\nr_{e_{\chi}E[G]}(e_{\chi} \hgi) U_{\chi} \fit( (\nabla/\delta(C))[\chi]_{G}) = \fit(\cl_{K}[\chi]^{G}).
\]
\end{remark}

\begin{proof}
We shall abuse notation by systematically ignoring $2$-parts.
We begin by following the proof of \cite[Lemma 8.2]{MR2371374}.
Recall that the map $\delta: C \rightarrow \nabla$ is injective (see \cite[\S 6]{MR2371374}) and that there is the crucial short exact sequence
$0 \rightarrow \cl_{K} \rightarrow \nabla \rightarrow \bar{\nabla} \rightarrow 0$.
By abuse of notation we also use $\delta$ for the map $C \rightarrow \bar{\nabla}$
and note that this is still injective since $C$ is free and thus torsion-free.
Since $(\nabla/\delta(C))^{-}$ is finite, we may choose a natural number $x$ such that
$x \nabla^{-} \subset \delta(C)^{-}$. Therefore we have two short exact sequences
\[
0 \rightarrow \cl_{K}^{-} \rightarrow \frac{\nabla^{-}}{\delta(C)^{-}} \rightarrow
\frac{\bar{\nabla}^{-}}{\delta(C)^{-}} \rightarrow 0, \qquad
0 \rightarrow \frac{\bar{\nabla}^{-}}{\delta(C)^{-}} \rightarrow
\frac{x^{-1}\delta(C)^{-}}{\delta(C)^{-}}
\rightarrow \frac{x^{-1}\delta(C)^{-}}{\bar{\nabla}^{-}} \rightarrow 0.
\]
These combine into the four term exact sequence
\[
0 \longrightarrow \cl_{K}^{-} \longrightarrow M_{1} \longrightarrow M_{2}
\longrightarrow M_{3} \longrightarrow 0,
\]
where
\[
M_{1} := \frac{\nabla^{-}}{\delta(C)^{-}}, \qquad
M_{2} := \frac{x^{-1}\delta(C)^{-}}{\delta(C)^{-}} , \qquad
M_{3} := \frac{x^{-1}\delta(C)^{-}}{\bar{\nabla}^{-}}.
\]
The functor $M \mapsto M[\chi] := T_{\chi} \otimes_{\Z} M$ is exact as $T_{\chi}$ is free
over $\Z$ and so we obtain the exact sequence
\[
0 \longrightarrow \cl_{K}[\chi] \longrightarrow M_{1}[\chi] \longrightarrow M_{2}[\chi]
\longrightarrow M_{3}[\chi] \longrightarrow 0.
\]
Since $x$ is a natural number and $\delta(C)^{-}$ is free, $M_{2}$ is of projective dimension $1$
over $R$ and so is $R$-c.t.\ Hence $M_{2}[\chi]$ is also $R$-c.t.\ and so we have the following commutative
diagram
\[
\xymatrix@1@!0@=36pt {
0 \ar@{->}[rr] & & \cl_{K}[\chi]^{G} \ar@{->}[rr] & &  M_{1}[\chi]^{G} \ar@{->}[rr]^{\alpha} & & M_{2}[\chi]^{G} \\
& & & &  M_{1}[\chi]_{G} \ar@{->}[u] \ar@{->}[rr]^{\beta} & & M_{2}[\chi]_{G} \ar@{->}[u]^{\simeq} \ar@{->}[rr] & &
M_{3} [\chi]_{G} \ar@{->}[rr] & & 0,
}
\]
where the rows are exact and the vertical maps are induced by $\Norm_{G}$.
If we identify $M_{2}[\chi]_{G}$ with $M_{2}[\chi]^{G}$ then $\im(\beta) \subseteq \im(\alpha)$.
Therefore we have
\begin{equation}\label{fit-contain}
\begin{array}{lll}
\fit(\cl_{K}[\chi]^{G}) &=& \fit(M_{1}[\chi]^{G}) \fit(\im(\alpha))^{-1} \\
& \supseteq &  \fit(M_{1}[\chi]^{G}) \fit(\im(\beta))^{-1} \\
& = & \fit(M_{1}[\chi]^{G})\fit(M_{2}[\chi]_{G})^{-1}\fit(M_{3}[\chi]_{G}).
\end{array}
\end{equation}
We now compute the two rightmost terms explicitly.
Let $n = |\mathcal{S}' \backslash \mathcal{S}_{\infty}|$. Then
\[
M_{2}[\chi]_{G} = T_{\chi} \otimes_{R} \frac{x^{-1}\delta(C)^{-}}{\delta(C)^{-}}
\cong T_{\chi} \otimes_{R} \left( \frac{x^{-1}R^{n}}{R^{n}} \right) \cong \frac{x^{-1}T_{\chi}^{n}}{T_{\chi}^{n}}
\cong  \frac{T_{\chi}^{n}}{xT_{\chi}^{n}},
\]
and so recalling that $T_{\chi}$ is locally free of rank $\chi(1)$ over $\mathcal{O}$ gives
\begin{equation}\label{M2-fit}
\fit(M_{2}[\chi]_{G}) = x^{n \chi(1)}\mathcal{O}.
\end{equation}
In \cite[bottom of p.1419]{MR2371374} it is noted that
\begin{equation}\label{nabla-minus}
\bar{\nabla}^{-} = \bigoplus_{\frakp \in \mathcal{S}' \backslash \mathcal{S}_{\infty}}
\left( \ind W^{0}_{\frakp} \right)^{-}.
\end{equation}
Recall from \cite[proof of Lemma 6.1]{MR2371374} that
$g_{\frakp} := |G_{0,\frakp}| + 1 - F_{\frakp}^{-1}$ maps to a non-zero divisor $\bar{g}$
of $\Z[G_{\frakp}/G_{0, \frakp}]$, and $g_{\frakp}^{-1}$ stands for any lift of $\bar{g}^{-1}$
to $\Q[G_{\frakp}]$. This uniquely defines the element $g_{\frakp}^{-1} \Norm_{G_{0,\frakp}}$
(note this is central in $\Q[G_{\frakp}]$). From \cite[top of p.1420]{MR2371374} we have
\begin{equation}\label{W-0}
W^{0}_{\frakp} =
\delta(x_{\frakp}) \cdot \langle 1, g_{\frakp}^{-1} \Norm_{G_{0,\frakp}} \rangle_{\Z[G_{\frakp}]}.
\end{equation}
Now recall that $h_{\frakp} := e'_{\frakp}g_{\frakp} + e_{\frakp}''$
(again this is central in $\Q[G_{\frakp}]$). In \cite[Lemma 8.3]{MR2371374}, it is shown that
we have
\begin{equation}\label{hpUp}
 \langle 1, g_{\frakp}^{-1} \Norm_{G_{0,\frakp}} \rangle_{\Z[G_{\frakp}]}
 = h_{\frakp}^{-1} \langle \Norm_{G_{0, \frakp}}, 1-e_{\frakp}'F_{\frakp}^{-1} \rangle_{\Z[G_{\frakp}]}
 = h_{\frakp}^{-1} U_{\frakp},
\end{equation}
(the proof is valid for the non-abelian case as the relevant elements are central in
$\Q[G_{\frakp}]$). Combining equations \eqref{nabla-minus}, \eqref{W-0} and \eqref{hpUp},
gives
\[
\bar{\nabla}^{-}
= \bigoplus_{\frakp \in \mathcal{S}' \backslash \mathcal{S}_{\infty}}
\left( \delta(x_{\frakp}) \ind h_{\frakp}^{-1}U_{\frakp} \right)^{-}
= \bigoplus_{\frakp \in \mathcal{S}' \backslash \mathcal{S}_{\infty}}
\left( \delta(x_{\frakp}) \Z[G] (h_{\frakp}^{-1}U_{\frakp}) \right)^{-},
\]
where the second equality follows from the fact that $\bar{\nabla}$ is torsion-free.
Hence
\[
M_{3}
= \frac{x^{-1}\delta(C)^{-}}{\bar{\nabla}^{-}}
\cong \bigoplus_{\frakp \in \mathcal{S}' \backslash \mathcal{S}_{\infty}}
\frac{R}{xR(h_{\frakp}^{-1}U_{\frakp})},
\]
and so
\[
M_{3}[\chi]_{G} =  T_{\chi} \otimes_{R} M_{3}
\cong
\bigoplus_{\frakp \in \mathcal{S}' \backslash \mathcal{S}_{\infty}}
\frac{T_{\chi}}{xT_{\chi}(h_{\frakp}^{-1}U_{\frakp})}
=
\bigoplus_{\frakp \in \mathcal{S}' \backslash \mathcal{S}_{\infty}}
\frac{T_{\chi}}{xT_{\chi}(e_{\chi}\mathfrak{M})(h_{\frakp}^{-1}U_{\frakp})}.
\]
Recall that $\hg := \prod_{\frakp \in \mathcal{S'} \backslash \mathcal{S}_{\infty}} h_{\frakp}$
and
\[
U_{\chi} := \prod_{\frakp \in \mathcal{S}_{\mathrm{ram}}(K/k)}
\nr_{e_{\chi}E[G]}(e_{\chi}\mathfrak{M}U_{\frakp})\mathcal{O} =
\prod_{\frakp \in \mathcal{S'} \backslash \mathcal{S}_{\infty}}
\nr_{e_{\chi}E[G]}(e_{\chi}\mathfrak{M}U_{\frakp})\mathcal{O},
\]
where the second equality holds since $U_{\frakp}=\Z[G_{\frakp}]$,
and hence
$\nr_{e_{\chi}E[G]}(e_{\chi}\mathfrak{M}U_{\frakp})\mathcal{O} =
\mathcal{O}$, if $\frakp \notin \mathcal{S}_{\mathrm{ram}}(K/k)$.
Note that $\Lambda := e_{\chi}\mathfrak{M}$ is a maximal
$\mathcal{O}$-order in $e_{\chi}E[G]$; so if $\mathcal{O}'$ is the
localisation of $\mathcal{O}$ at any prime ideal then $\Lambda' :=
\mathcal{O}' \otimes_{\mathcal{O}} \Lambda$ is a maximal
$\mathcal{O}'$-order and every (left) $\Lambda'$-ideal is principal
(see \cite[Theorem 18.7(ii)]{MR1972204}). In particular, for each
$\frakp \in \mathcal{S'} \backslash \mathcal{S}_{\infty}$ there
exists an element $y_{\frakp}$ such that
$\Lambda' \otimes_{\Lambda} x(e_{\chi}\mathfrak{M})(h_{\frakp}^{-1}U_{\frakp}) 
= \Lambda' y_{\frakp}$ 
and therefore an exact sequence of $\mathcal{O}'$-modules of the form
\begin{equation}\label{eqn:local-ES}
\mathcal{O}' \otimes_{\mathcal{O}} T_\chi \stackrel{y_{\frakp}} \longrightarrow
\mathcal{O}' \otimes_{\mathcal{O}} T_\chi \longrightarrow
\mathcal{O}' \otimes_{\mathcal{O}} 
\frac{T_{\chi}}{xT_{\chi}(e_{\chi}\mathfrak{M})(h_{\frakp}^{-1}U_{\frakp})}
\longrightarrow 0.
\end{equation}
Now $\mathcal{O}'\otimes_\mathcal{O}T_{\chi}$ is free of rank
$\chi(1)$ over $\mathcal{O}'$. Thus, since $\nr_{e_{\chi}E[G]}$ is
the determinant map $e_{\chi}E[G] \cong \Mat_{\chi(1)}(E)
\rightarrow E$, the definition of Fitting ideal combines with \eqref{eqn:local-ES}
and our definition of the fractional ideal
$\nr_{e_{\chi}E[G]}(x(e_{\chi}\mathfrak{M})(h_{\frakp}^{-1}U_{\frakp}))\mathcal{O}$
to imply that
\begin{eqnarray*}
\mathrm{Fit}_{\mathcal{O}'}\left(\mathcal{O}'\otimes_\mathcal{O}\frac{T_{\chi}}{xT_{\chi}(e_{\chi}\mathfrak{M})(h_{\frakp}^{-1}U_{\frakp})}\right)
&=& \mathcal{O}' \cdot \det_{\mathcal{O}'}(y_{\frakp})\\
&=&
\mathcal{O}'\otimes_\mathcal{O}\nr_{e_{\chi}E[G]}(x(e_{\chi}\mathfrak{M})(h_{\frakp}^{-1}U_{\frakp}))\mathcal{O}.
\end{eqnarray*}
However, Fitting ideals over $\mathcal{O}$ can be computed by
localising and so, by taking the product over all primes $\frakp
\in \mathcal{S'} \backslash \mathcal{S}_{\infty}$, we therefore
have
\begin{equation}\label{M3-fit}
\fit(M_{3}[\chi]_{G})
= \prod_{\frakp \in \mathcal{S}' \backslash \mathcal{S}_{\infty}}
\nr_{e_{\chi}E[G]}(x(e_{\chi}\mathfrak{M})(h_{\frakp}^{-1}U_{\frakp}))\mathcal{O}\\
= x^{n\chi(1)} \nr_{e_{\chi}E[G]}(e_{\chi} \hgi) U_{\chi}.
\end{equation}
Substituting equations \eqref{M2-fit} and \eqref{M3-fit} into the containment \eqref{fit-contain} gives
\begin{equation}\label{eqn:G-up-contain}
\nr_{e_{\chi}E[G]}(e_{\chi}\hgi)  U_{\chi} \fit( M_{1}[\chi]^{G})
\subseteq  \fit(\cl_{K}[\chi]^{G}).
\end{equation}
We have an exact sequence
\begin{equation*}
0 \longrightarrow \widehat{H}^{-1}(G, M_{1}[\chi]) \longrightarrow M_{1}[\chi]_{G}
\longrightarrow M_{1}[\chi]^{G} \longrightarrow \widehat{H}^{0}(G, M_{1}[\chi]) \longrightarrow 0,
\end{equation*}
where $\widehat{H}^{i}(G, M)$ denotes Tate cohomology. Hence
\begin{equation}\label{eqn:herbrand-quot}
\fit(M_{1}[\chi]^{G})
= \fit(M_{1}[\chi]_{G}) \frac{\fit(\widehat{H}^{0}(G, M_{1}[\chi]))}{\fit(\widehat{H}^{-1}(G, M_{1}[\chi]))}.
\end{equation}
Recall that the top row of (D2)$^{-}$ is an exact sequence of f.g.\ $R$-modules
\[
0 \longrightarrow \mu_{K}^{-} \longrightarrow A^{-} \longrightarrow \tilde{B}^{-}
\longrightarrow M_{1} \longrightarrow 0,
\]
where $A^{-}$ and $\tilde{B}^{-}$ are $R$-c.t.\ Hence
\begin{equation}\label{eqn:co-shift}
\widehat{H}^{0}(G, M_{1}[\chi]) \cong {H}^{2}(G, \mu_{K}[\chi]) \quad \textrm{ and } \quad
\widehat{H}^{-1}(G, M_{1}[\chi]) \cong {H}^{1}(G, \mu_{K}[\chi]).
\end{equation}
Combining \eqref{eqn:G-up-contain}, \eqref{eqn:herbrand-quot} and \eqref{eqn:co-shift}
now gives the desired result.
\end{proof}

\section{Fitting ideals from refined Euler characteristics}\label{sec:fit-euler}

Let $I_{\mathcal{O}}$ denote the multiplicative group of invertible $\mathcal{O}$-modules in
$\C$. There exists a natural isomorphism
$\iota:K_{0}(\mathcal{O}, \C) \stackrel{\sim}\rightarrow I_{\mathcal{O}}$ with
$\iota((P,\tau, Q))
=\tilde{\tau}(\det_{\mathcal{O}}(P) \otimes_{\mathcal{O}} \det_{\mathcal{O}}(Q)^{-1})$
where $\tilde{\tau}$ is the isomorphism
\[
\C \otimes_{\mathcal{O}}
(\det_{\mathcal{O}}(P) \otimes_{\mathcal{O}} \det_{\mathcal{O}}(Q)^{-1})
\cong
\det_{\C}(\C \otimes_{\mathcal{O}} Q) \otimes_{\C}
\det_{\C}(\C \otimes_{\mathcal{O}} Q)^{-1}
\cong \C
\]
induced by $\tau$. Indeed, $\iota$ is induced by the exact sequence
\[
K_{1}(\mathcal{O}) \longrightarrow K_{1}(\C) \longrightarrow K_{0}(\mathcal{O}, \C)
\longrightarrow K_{0}(\mathcal{O}) \longrightarrow K_{0}(\C)
\]
and the canonical isomorphisms $K_{1}(\C) \stackrel{\sim} \rightarrow \C^{\times}$
and $K_{1}(\mathcal{O}) \stackrel{\sim} \rightarrow \mathcal{O}^{\times}$.
Under this identification, the boundary map
$\C^{\times} \cong K_{1}(\C) \rightarrow K_{0}(\mathcal{O}, \C) \cong I_{\mathcal{O}}$
simply sends $x$ to the lattice $x\mathcal{O}$.

In what follows, $\varphi_{\mathrm{triv}}$ denotes the only metrisation possible, namely the unique isomorphism from the complex vector space $0$ to itself.

\begin{lemma}\label{lemma:euler-to-fit}
If $0 \rightarrow U \rightarrow A \rightarrow B \rightarrow V \rightarrow 0$
is an exact sequence of f.g.\
$\mathcal{O}$-modules with $U$ and $V$ finite then
$\iota(\euler(A \rightarrow B, \varphi_{\mathrm{triv}})) = \fit(U)^{-1}\fit(V)$.
\end{lemma}

\begin{proof}
We have a distinguished triangle of perfect metrised complexes of $\mathcal{O}$-modules
\[
\mathcal{C}_{0} \longrightarrow \mathcal{C}_{1} \longrightarrow \mathcal{C}_{2}
\longrightarrow \mathcal{C}_{0}[1]
\]
with $\mathcal{C}_{0}$ the complex $U[0]$, $\mathcal{C}_{2}$ the complex $V[-1]$
and $\mathcal{C}_{1}$ the complex $A \rightarrow B$ with the first term placed in degree $0$.
(To see this, write $\widetilde{\mathcal{C}}_{2}$ for the complex $A/U \rightarrow B$ with first term placed in degree zero and map induced by $A \rightarrow B$:
then $\widetilde{\mathcal{C}}_{2}$ is naturally quasi-isomorphic to $\mathcal{C}_{2}$
and also lies in the obvious short exact sequence of complexes of the form
\[
0 \longrightarrow \mathcal{C}_{0} \longrightarrow \mathcal{C}_{1}
\longrightarrow \widetilde{\mathcal{C}}_{2} \longrightarrow 0 .)
\]
Furthermore, since $\mathcal{O}$ is a Dedekind domain
every f.g.\ $\mathcal{O}$-module is of projective dimension at most one.
Therefore the refined Euler characteristics in question are additive (see \cite[Theorem 2.8]{MR2076565}),
and so we are reduced to showing that for a finite $\mathcal{O}$-module $M$ we have
\[
\iota(\euler(M[i], \varphi_{\mathrm{triv}})) = \fit(M)^{(-1)^{i+1}}.
\]
However, $\euler(M[i], \varphi_{\mathrm{triv}}) = (-1)^{i}\euler(M[0], \varphi_{\mathrm{triv}})$
and so we are further reduced to considering the case $i=0$.
There exists an exact sequence of f.g.\ $\mathcal{O}$-modules
\[
0 \longrightarrow P \stackrel{d} \longrightarrow F \longrightarrow M \longrightarrow 0
\]
with $P$ projective and $F$ free of equal rank $r$. We fix an isomorphism $F \cong \mathcal{O}^{r}$
and hence an identification of $\det_{\mathcal{O}}(F) = \wedge_{\mathcal{O}}^{r}(F)$ with $\mathcal{O}$.
Under this identification, $\fit(M)$ is by definition the image of the homomorphism
$\wedge_{\mathcal{O}}^{r}(d):\wedge_{\mathcal{O}}^{r}(P) \rightarrow \wedge_{\mathcal{O}}^{r}(F)=\mathcal{O}$. Setting $\tau := \C \otimes_{\mathcal{O}} d$ and noting that $M[0]$ is quasi-isomorphic to the complex $P \rightarrow F$ with the second term placed in degree $0$, we therefore have
\begin{eqnarray*}
\iota(\euler(M[0], \varphi_{\mathrm{triv}})) &=&
\iota((F, \tau^{-1}, P)) = \iota((P, \tau, F))^{-1} =
\tilde{\tau}(\det_{\mathcal{O}}(P) \otimes_{\mathcal{O}} \det_{\mathcal{O}}(F)^{-1})^{-1}\\
&=&  \im(\wedge_{\mathcal{O}}^{r}(d))^{-1} = \fit(M)^{-1},
\end{eqnarray*}
as required.
\end{proof}

\section{Two annihilation lemmas}

Let $\chi$ be an irreducible character of a finite group $G$
and let $M$ be a $\Z[G]$-module.

\begin{lemma}\label{useful}
If $x \in \Ann_{\mathcal{O}}(M[\chi]^G)$,
then $x \cdot \pr_{\chi} \in \Ann_{\mathcal{O}[G]}(\mathcal{O} \otimes_{\Z} M)$.
\end{lemma}

\begin{proof}
It suffices to show that the result holds after localising at $\frakp$ for all primes
$\frakp$ of $\mathcal{O}$. We set $n := \chi(1)$ and recall that $T_{\chi}$ is locally free
of rank $n$ over $\mathcal{O}$. In what follows, we abuse notation by omitting subscripts $\frakp$
(i.e. all $\mathcal{O}$-modules below are localised at $\frakp$).

We fix an $\mathcal{O}$-basis $\{t_i: 1 \le i \le n\}$ of $T_\chi$ and
write $\rho_{\chi}: G \to {\rm GL}_n(\mathcal{O})$ for the associated representation.
Then for each $m \in M$ and each index $i$ the element
$T_i(m) := \sum_{g \in G}g(t_i\otimes m)$
belongs to $M[\chi]^G = (T_{\chi} \otimes_{\Z} M)^{G}$. 
Now in 
$M[\chi] = T_{\chi} \otimes_{\Z} M = T_{\chi} \otimes_{\mathcal{O}} (\mathcal{O} \otimes_{\Z}M)$
we have
\begin{eqnarray*}
T_i(m)
&=& \sum_{g \in G} t_i g^{-1} \otimes g(m) 
= \sum_{g \in G} \sum_{j=1}^{j=n}\rho_{\chi}(g^{-1})_{ij}t_j \otimes g(m) \\
&=& \sum_{g \in G} \sum_{j=1}^{j=n}\rho_{\overline{\chi}}(g)_{ji}t_j\otimes g(m) 
= \sum_{j = 1}^{j=n}t_j\otimes \left(\sum_{g \in G}\rho_{\overline{\chi}}(g)_{ji} g(m)\right).
\end{eqnarray*}
However, $x$ annihilates $T_i(m) \in M[\chi]^G$ and $\{t_{j}: 1 \le j \le n\}$ is an $\mathcal{O}$-basis of $T_\chi$, so the above equation implies that
$x \cdot \sum_{g \in G} \rho_{\overline{\chi}}(g)_{ji}g(m) =0$ for all $i$ and $j$.
Hence each element
$c(x)_{ij} := x \cdot \sum_{g \in G}\rho_{\overline{\chi}}(g)_{ji}g$
belongs to $\Ann_{\mathcal{O}[G]}(\mathcal{O} \otimes_{\Z} M)$.
In particular, the element
\[
\sum_{i = 1}^{i = n}c(x)_{ii}
= \sum_{i = 1}^{i = n} x \cdot \sum_{g \in G} \rho_{\overline{\chi}} (g)_{ii}g
= x \cdot \sum_{g \in G} \left(\sum_{i = 1}^{i = n}\rho_{\overline{\chi}}(g)_{ii}\right)g
= x \cdot \sum_{g \in G}\overline{\chi}(g)g
= x \cdot \pr_{\chi}
\]
belongs to $\Ann_{\mathcal{O}[G]}(\mathcal{O} \otimes_{\Z} M)$, as required.
\end{proof}

\begin{lemma}\label{useful2}
If $x \in \mathcal{O}$ such that
$x \cdot \pr_{\chi} \in \Ann_{\mathcal{O}[G]}(\mathcal{O} \otimes_{\Z} M)$,
then for every $y \in \mathcal{D}_{E/\Q}^{-1}$ we have
$
\sum_{\omega \in \Gal(E/\Q)}y^{\omega} x^{\omega} \cdot \pr_{\chi^{\omega}} \in \Ann_{\Z [G]}(M).
$
\end{lemma}

\begin{proof}
The hypotheses imply that $y x \cdot \pr_{\chi}$ belongs to
\[
\mathcal{D}_{E/\Q}^{-1}\cdot \Ann_{\mathcal{O}[G]}(\mathcal{O} \otimes_{\Z} M)
 = \mathcal{D}_{E/\Q}^{-1} \otimes_{\Z} \Ann_{\Z [G]}(M).
\]
The element
\[
\sum_{\omega \in \Gal(E/\Q)} y^{\omega} x^{\omega} \cdot \pr_{\chi^{\omega}} =
\sum_{\omega \in \Gal(E/\Q)}(y x \cdot \pr_{\chi} )^{\omega}
\]
therefore belongs to
$\tr_{E/\Q}(\mathcal{D}_{E/\Q}^{-1}) \otimes_{\Z} \Ann_{\Z[G]}(M) \subseteq \Ann_{\Z [G]}(M)$,
as required.
\end{proof}

\section{Proofs of the main results}\label{sec:main-results}

\begin{proof}[Proof of Theorem \ref{main-theorem}]
Let $\phi$ be the character of $\Gal(K/k)$ whose inflation to $G=\Gal(L/k)$ is $\chi$.
For each $x \in \cl_{L}$, we have
$\pr_{\chi}(x) = \pr_{\phi} (\Norm_{\Gal(L/K)}(x))$
and $\Norm_{\Gal(L/K)}(x) \in \cl_{K}$.
However, we also have $L(s, \chi) = L(s, \phi)$ (see \cite[\S 4.2]{MR782485}) and so 
we are therefore reduced to the case
$L=K$. Note that $\chi=\phi$ remains irreducible.

We now repeat a reduction argument given in \cite[top of p.71]{MR782485}.
The order of vanishing of
$L(s,\chi)=L_{\mathcal{S}_{\infty}}(s, \chi)$ at $s=0$ is given by
\begin{equation}\label{eqn:order-of-vanishing}
r_{\mathcal{S}_{\infty}}(\chi) =
 \sum_{v \in \mathcal{S}_{\infty}} \dim_{\C} V_{\chi}^{G_{v}} - \dim_{\C} V_{\chi}^{G},
\end{equation}
where $V_{\chi}$ is a $\C[G]$-module with character $\chi$
(see \cite[Chapter I, Proposition 3.4]{MR782485}).
If $r_{\mathcal{S}_{\infty}}(\chi)>0$ then $L(0, \chi)=0$ and so the result is trivial.
Hence we may suppose that $r_{\mathcal{S}_{\infty}}(\chi)=0$. Since $\chi$ is non-trivial, we have
$V_{\chi}^{G} = \{ 0 \}$ and so \eqref{eqn:order-of-vanishing} gives
$V_{\chi}^{G_{v}}= \{ 0 \}$ for each $v \in \mathcal{S}_{\infty}$.
In particular, $G_{v}$ is non-trivial for $v \in \mathcal{S}_{\infty}$ so
$k$ is totally real and $K$ is totally complex. Now $G_{v} = \{ 1, j_{w} \}$ for a complex place
$w$ of $K$ and $j_{w}$ acts as $-1$ on $V_{\chi}$ since $j_{w}^{2}=1$
and $V_{\chi}^{j_{w}} = \{ 0 \}$. Thus, since the representation $V_{\chi}$
is faithful, all the $j_{w}$ are equal to the same $j \in G$. Hence $K$ is a totally
imaginary quadratic extension of the totally real subfield $K^{\langle j \rangle}$,
i.e., $K$ is a CM field. Furthermore, $\chi$ is odd because $j$ acts as $-1$ on $V_{\chi}$.

For the rest of this proof, we abuse notation and only consider $p$-parts.
Recall from Proposition \ref{prop:explicit-eqn}
that under the assumption that ETNC holds for the motive $h^{0}(K)$ with coefficients
in $R$, we have
\begin{equation}\label{eqn:explicit}
\hat{\partial}(L(0)^{\#} \nr_{\R[G]}(\hg))^{-} = X_{\infty}^{-} \textrm{ in } K_{0}(R, \R[G]^{-}).
\end{equation}
As $\chi$ is odd, base change gives a natural homomorphism
\[
\mu_{\chi} : K_{0}(R, \R[G]^{-}) \rightarrow  K_{0}(\mathcal{O}, \C), \quad
(P,f,Q) \mapsto (T_{\chi} \otimes_{R} P, \mathrm{id} \otimes f, T_{\chi} \otimes_{R} Q)
= (P[\chi]_{G}, f_{\chi}, Q[\chi]_{G} ).
\]
Under condition (\textasteriskcentered) it can be shown that
the Strong Stark Conjecture at $p$ for $\chi$ follows from Wiles' proof
of the Main conjecture for totally real fields
(for details see \cite[Corollary 2, p.24]{nickel-ETNC}, for example).
Since the Strong Stark Conjecture can be interpreted
as `ETNC modulo torsion' and $K_{0}(\mathcal{O}, \C)$ is torsion-free,
the image under $\mu_{\chi}$ of equation \eqref{eqn:explicit} holds under
our hypotheses, i.e.,
\[
\mu_{\chi} (\hat{\partial}(L(0)^{\#} \nr_{\R[G]}(\hg))^{-})
= \mu_{\chi}(X_{\infty}^{-})
\textrm{ in } K_{0}(\mathcal{O}, \C).
\]
Since $\mu_{\chi}$ factors via $K_{0}(\Z[G], \C[G])$ and $K_{0}(\mathcal{O},\C)$ is torsion-free,
Lemma \ref{lemma:commutes-up-to-2} then gives
\begin{equation}\label{eqn:mu-chi-eq}
(\mathcal{O}, \nr_{e_{\chi}E[G]}(e_{\chi} \hg) L(0,\bar{\chi}),\mathcal{O})
= \mu_{\chi}(X_{\infty}^{-}) \textrm{ in } K_{0}(\mathcal{O}, \C).
\end{equation}
Recall that the top row of (D2)$^{-}$ is an exact sequence of f.g.\ $R$-modules
\begin{equation}\label{eqn:top-row-of-D2-minus}
0 \longrightarrow \mu_{K}^{-} \longrightarrow A^{-}
\longrightarrow \tilde{B}^{-} \longrightarrow (\nabla/\delta(C))^{-} \longrightarrow 0.
\end{equation}
This gives the commutative diagram
\[
\xymatrix@1@!0@=36pt {
0 \ar@{->}[rr] & & \mu[\chi]^{G} \ar@{->}[rr] & & A[\chi]^{G} \ar@{->}[rr] & &
\tilde{B}[\chi]^{G} \\
& & & &  A[\chi]_{G} \ar@{->}[u]^{\simeq} \ar@{->}[rr] & & \tilde{B}[\chi]_{G} \ar@{->}[u]^{\simeq} \ar@{->}[rr] & &
(\nabla/\delta(C))[\chi]_{G} \ar@{->}[rr] & & 0,
}
\]
where the vertical maps induced by $\Norm_{G}$ are isomorphisms since
$A^{-}$ and $\tilde{B}^{-}$ are $R$-c.t. Hence we have an exact sequence
\begin{equation}\label{eqn:combined-seq-theta-chi}
0 \longrightarrow \mu[\chi]^{G} \longrightarrow A[\chi]_{G}
\longrightarrow \tilde{B}[\chi]_{G} \longrightarrow (\nabla/\delta(C))[\chi]_{G}
\longrightarrow 0.
\end{equation}

Now recall that $X_{\infty}^{-}$ is the refined Euler characteristic of
\eqref{eqn:top-row-of-D2-minus}.
Since $\mu_{K}^{-}$ and $(\nabla/\delta(C))^{-}$ are finite, the only possible metrisation is
$\varphi_{\mathrm{triv}}$ (i.e. $0 \stackrel{\sim}\longrightarrow 0$) and so
\eqref{eqn:mu-chi-eq} becomes
\begin{equation}\label{eqn:equal-in-K(0,C)}
(\mathcal{O},  \nr_{e_{\chi}E[G]}(e_{\chi} \hg)L(0,\bar{\chi}), \mathcal{O})
=  \euler(A[\chi]_{G} \rightarrow \tilde{B}[\chi]_{G}, \varphi_{\mathrm{triv}})
\textrm{ in } K_{0}(\mathcal{O}, \C).
\end{equation}
By Lemma \ref{lemma:euler-to-fit}, \eqref{eqn:combined-seq-theta-chi} and \eqref{eqn:equal-in-K(0,C)} give an equality of $\mathcal{O}$-lattices of the form
\[
\nr_{e_{\chi}E[G]}(e_{\chi} \hg) L(0, \bar{\chi}) \mathcal{O}
= \fit((\mu_{K}[\chi])^{G})^{-1} \fit( (\nabla/\delta(C))[\chi]_{G}) .
\]
Combining this equality with Proposition \ref{prop:fit-comp} gives
\[
L(0, \bar{\chi}) U_{\chi}  \prod_{i=0}^{i=2} \fit(H^{i}(G, \mu_{K}[\chi]))^{(-1)^{i}}
\subseteq \fit(\cl_{K}^{-}[\chi]^{G}).
\]
Recalling that $h(\mu_{K},\chi) := \prod_{i=0}^{i=2} \fit(H^{i}(G, \mu_{K}[\chi]))^{(-1)^{i}}$,
we then obtain
\[
L(0, \bar{\chi}) U_{\chi} h(\mu_{K},\chi)
\subseteq \fit(\cl_{K}^{-}[\chi]^{G})
\subseteq \Ann_{\mathcal{O}}(\cl_{K}^{-}[\chi]^{G}).
\]
Hence for any $x \in U_{\chi} \cdot h(\mu_{K},\chi)$, Lemma \ref{useful} implies that
\begin{equation}\label{eqn:L-annihilates}
x L(0,\bar{\chi}) \cdot \pr_{\chi} \in
\Ann_{\mathcal{O}[G]^-}(\mathcal{O} \otimes_{\Z} \cl^{-}_{K}) \subseteq
\Ann_{\mathcal{O}[G]}(\mathcal{O} \otimes_{\Z} \cl_{K}).
\end{equation}
The desired result now follows by applying Lemma \ref{useful2}.
\end{proof}

\begin{proof}[Proof of Corollary \ref{cor:weak-brumer}]
Let $\chi$ be a non-trivial irreducible character of $G$ and let $K:=L^{\ker(\chi)}$.
As every inertia subgroup is normal in $G$, every inertia subgroup of $\Gal(K/k)$ is normal.
Choose $E_{\chi}$ such that $d_{\chi}=[E_{\chi}:\Q(\chi)]$ and
let $G \cdot \chi$ denote the orbit of $\chi$ in $\mathrm{Irr}(G)$. 
Then taking into account Remark \ref{rmk:simplifying-thm} and applying Theorem \ref{main-theorem} 
with $x=1$ shows that
\[
\sum_{\omega \in \Gal(E_{\chi}/\Q)} \!\!\! L(0, \bar{\chi}^{\omega})
\cdot \pr_{\chi^{\omega} }
= d_{\chi} \sum_{\psi \in G \cdot \chi} \!\!\! L(0, \bar{\psi})
\cdot \pr_{\psi}
\]
belongs to the centre of $\Z_{(p)}[G]$ and annihilates $\Z_{(p)}  \otimes_{\Z} \cl_L$.
Hence summing over all non-trivial irreducible characters of $G$ gives the desired result in the case
$\mathcal{S} = \mathcal{S}_{\infty}$.

If $\mathcal{S} \supsetneqq \mathcal{S}_{\infty}$ then $L_{\mathcal{S}}(0,\chi)$
is $L(0,\chi)$ multiplied by factors of the form
\[
L_{K_{\Frakp}/k_{\frakp}}(0, \psi)^{-1}
= \lim_{s \to 0} \det_{\C}(1-F_{\frakp} (\mathrm{N} \frakp)^{-s} \mid V_{\psi}^{G_{0,\frakp}}),
\]
each of which is an element of $\mathcal{O}$ (possibly zero). Hence the containment \eqref{eqn:L-annihilates} is still valid when $L(0,\chi)$ is replaced by
$L_{\mathcal{S}}(0,\chi)$, giving the analogous version of Theorem \ref{main-theorem}. 
The desired result then follows by the same argument as above.
\end{proof}

\section{Acknowledgment}

We would like to thank Cornelius Greither for helpful discussions regarding his paper \cite{MR2371374}.

\bibliography{nonabstickbib}

\providecommand{\bysame}{\leavevmode\hbox to3em{\hrulefill}\thinspace}
\providecommand{\MR}{\relax\ifhmode\unskip\space\fi MR }
\providecommand{\MRhref}[2]{%
  \href{http://www.ams.org/mathscinet-getitem?mr=#1}{#2}
}
\providecommand{\href}[2]{#2}
\begin{thebibliography}{RW97}

\bibitem[BB07]{MR2371375}
M.~Breuning and D.~Burns, \emph{Leading terms of {A}rtin {$L$}-functions at
  {$s=0$} and {$s=1$}}, Compos. Math. \textbf{143} (2007), no.~6, 1427--1464.
  \MR{2371375 (2009a:11232)}

\bibitem[BF01]{MR1884523}
D.~Burns and M.~Flach, \emph{Tamagawa numbers for motives with
  (non-commutative) coefficients}, Doc. Math. \textbf{6} (2001), 501--570
  (electronic). \MR{1884523 (2002m:11055)}

\bibitem[BF03]{MR1981031}
\bysame, \emph{Tamagawa numbers for motives with (noncommutative) coefficients
  {II}}, Amer. J. Math. \textbf{125} (2003), no.~3, 475--512. \MR{1981031
  (2004c:11111)}

\bibitem[Ble06]{MR2226270}
W.~Bley, \emph{Equivariant {T}amagawa number conjecture for abelian extensions
  of a quadratic imaginary field}, Doc. Math. \textbf{11} (2006), 73--118
  (electronic). \MR{2226270 (2007a:11153)}

\bibitem[Bre04]{breuning-thesis}
M.~Breuning, \emph{Equivariant epsilon constants for {G}alois extensions of
  number fields and $p$-adic fields}, Ph.D. thesis, King's College London,
  2004.

\bibitem[Bur01]{MR1863302}
D.~Burns, \emph{Equivariant {T}amagawa numbers and {G}alois module theory {I}},
  Compos. Math. \textbf{129} (2001), no.~2, 203--237. \MR{1863302
  (2002g:11152)}

\bibitem[Bur04]{MR2076565}
\bysame, \emph{Equivariant {W}hitehead torsion and refined {E}uler
  characteristics}, Number theory, CRM Proc. Lecture Notes, vol.~36, Amer.
  Math. Soc., Providence, RI, 2004, pp.~35--59. \MR{2076565 (2005d:19002)}

\bibitem[Bur08]{MR2443986}
\bysame, \emph{On refined {S}tark conjectures in the non-abelian case}, Math.
  Res. Lett. \textbf{15} (2008), no.~5, 841--856. \MR{2443986}

\bibitem[Bur09]{DerLfncs}
\bysame, \emph{On derivatives of {A}rtin {$L$}-series}, preprint, 2009.

\bibitem[Chi83]{MR724009}
T.~Chinburg, \emph{On the {G}alois structure of algebraic integers and
  {$S$}-units}, Invent. Math. \textbf{74} (1983), no.~3, 321--349. \MR{724009
  (86c:11096)}

\bibitem[CN79]{MR524276}
P.~Cassou-Nogu{\`e}s, \emph{Valeurs aux entiers n\'egatifs des fonctions z\^eta
  et fonctions z\^eta {$p$}-adiques}, Invent. Math. \textbf{51} (1979), no.~1,
  29--59. \MR{MR524276 (80h:12009b)}

\bibitem[CR81]{MR632548}
C.~W. Curtis and I.~Reiner, \emph{Methods of representation theory. {V}ol.
  {I}}, Pure and Applied Mathematics, John Wiley \& Sons Inc., New York, 1981,
  With applications to finite groups and orders, A Wiley-Interscience
  Publication. \MR{632548 (82i:20001)}

\bibitem[CR87]{MR892316}
\bysame, \emph{Methods of representation theory. {V}ol. {II}}, Pure and Applied
  Mathematics, John Wiley \& Sons Inc., New York, 1987, With applications to
  finite groups and orders, A Wiley-Interscience Publication. \MR{892316
  (88f:20002)}

\bibitem[DR80]{MR579702}
P.~Deligne and K.~A. Ribet, \emph{Values of abelian {$L$}-functions at negative
  integers over totally real fields}, Invent. Math. \textbf{59} (1980), no.~3,
  227--286. \MR{579702 (81m:12019)}

\bibitem[Gre04]{MR2088712}
C.~Greither, \emph{Arithmetic annihilators and {S}tark-type conjectures},
  Stark's conjectures: recent work and new directions, Contemp. Math., vol.
  358, Amer. Math. Soc., Providence, RI, 2004, pp.~55--78. \MR{2088712
  (2005h:11259)}

\bibitem[Gre07]{MR2371374}
\bysame, \emph{Determining {F}itting ideals of minus class groups via the
  equivariant {T}amagawa number conjecture}, Compos. Math. \textbf{143} (2007),
  no.~6, 1399--1426. \MR{2371374 (2009a:11226)}

\bibitem[Nic]{nickel-ETNC}
A.~Nickel, \emph{On the {E}quivariant {T}amagawa {N}umber {C}onjecture in tame
  {CM}-extensions}, to appear in Math. Z.

\bibitem[Nic09]{nickel-LRNC}
\bysame, \emph{The {L}ifted {R}oot {N}umber {C}onjecture for small sets of
  places}, J. Lond. Math. Soc. \textbf{80} (2009), no.~2, 446--470.

\bibitem[Rei03]{MR1972204}
I.~Reiner, \emph{Maximal orders}, London Mathematical Society Monographs. New
  Series, vol.~28, The Clarendon Press Oxford University Press, Oxford, 2003,
  Corrected reprint of the 1975 original, With a foreword by M. J.\ Taylor.
  \MR{1972204 (2004c:16026)}

\bibitem[RW96]{MR1394524}
J.~Ritter and A.~Weiss, \emph{A {T}ate sequence for global units}, Compos.
  Math. \textbf{102} (1996), no.~2, 147--178. \MR{1394524 (97d:11170)}

\bibitem[RW97]{MR1423032}
\bysame, \emph{Cohomology of units and {$L$}-values at zero}, J. Amer. Math.
  Soc. \textbf{10} (1997), no.~3, 513--552. \MR{1423032 (98a:11150)}

\bibitem[Tat66]{MR0207680}
J.~Tate, \emph{The cohomology groups of tori in finite {G}alois extensions of
  number fields}, Nagoya Math. J. \textbf{27} (1966), 709--719. \MR{0207680 (34
  \#7495)}

\bibitem[Tat84]{MR782485}
\bysame, \emph{Les conjectures de {S}tark sur les fonctions {$L$} d'{A}rtin en
  {$s=0$}}, Progress in Mathematics, vol.~47, Birkh\"auser Boston Inc., Boston,
  MA, 1984, Lecture notes edited by Dominique Bernardi and Norbert Schappacher.
  \MR{782485 (86e:11112)}

\end{thebibliography}
\bibliographystyle{amsalpha}

\end{document}